\documentclass[11pt]{amsart}
\usepackage[T1]{fontenc}
\usepackage{txfonts}
\usepackage{times}
\usepackage{amssymb,verbatim}
\usepackage[all]{xy}
\usepackage{subfig}
\usepackage{tikz}
\usepackage{color}
\usepackage{a4wide}
\usepackage[normalem]{ulem}
\usepackage{hyperref}
\usepackage{wrapfig}
\usetikzlibrary{arrows}

\newcommand\QQQ{\mathcal Q}
\newcommand{\R}{\mathbb{R}}
\newcommand{\C}{\mathbb{C}}
\renewcommand\P{\mathbb{P}}
\newcommand\Z{\mathbb{Z}}
\newcommand\Prym{\mathrm{Prym}}
\newcommand{\N}{\mathbb{N}}
\newcommand{\Q}{\mathbb{Q}}
\renewcommand{\H}{\Omega \mathcal{M}}
\newcommand{\SL}{{\rm SL}}
\newcommand{\GL}{{\rm GL}}
\newcommand{\PrD}{\Omega E_D}

\newcommand{\fl}{{\rm flux}}
\newcommand{\Fl}{{\rm Flux}}
\newcommand{\dist}{\mathbf{d}}
\newcommand{\Ord}{\mathcal{O}}

\newcommand{\eps}{\varepsilon}

\newcommand{\Id}{\mathrm{Id}}

\newcommand{\Alp}{\mathcal{A}}
\newcommand{\lbd}{\lambdaup}
\newcommand{\vr}{\rho}
\newcommand{\inv}{\tauup}
\newcommand{\ol}{\overline}

\newcommand{\ra}{\rightarrow}
\newcommand{\comp}{\star}

\newtheorem{Theorem}{Theorem}[section]
\newtheorem{MainThm}{Theorem}

\newtheorem{Corollary}[Theorem]{Corollary}
\newtheorem{Lemma}[Theorem]{Lemma}
\newtheorem{Proposition}[Theorem]{Proposition}
\newtheorem{Remark}[Theorem]{Remark}
\newtheorem{Definition}[Theorem]{Definition}
\newtheorem{Claim}[Theorem]{Claim}
\newtheorem{Example}[Theorem]{Example}
\begin{document}
\title{Complete periodicity of Prym eigenforms}
\author{Erwan Lanneau and Duc-Manh Nguyen}
\date{\today}

\address{
Institut Fourier, Universit\'e de Grenoble I, BP 74, 38402 Saint-Martin-d'H\`eres, France
}
\email{erwan.lanneau@ujf-grenoble.fr}
\address{
IMB Bordeaux-Universit\'e Bordeaux 1, 351, Cours de la Lib\'eration, 33405 Talence Cedex, France
}
\email{duc-manh.nguyen@math.u-bordeaux1.fr}

\renewcommand{\abstractname}{R\'esum\'e}
\begin{abstract}
Dans cet article nous d\'emontrons plusieurs r\'esultats topologiques sur les 
formes propres des lieux Prym, formes differentielles ab\'eliennes d\'ecouvertes par McMullen 
dans des travaux ant\'erieurs. Nous obtenons une propri\'et\'es d\^ites de compl\`ete 
periodicit\'e (introduite par Calta), ainsi que de nouvelles familles de surfaces de translation v\'erifiant la 
dichotomie topologique de Veech.
Comme cons\'equences nous montrons que l'ensemble limite des groupes 
de Veech de formes propres de certaines strates en genre $3,4$, et $5$ est soit vide, soit un point, soit tout le cercle \`a 
l'infini. Ceci nous permet de plus de construire de nouveaux exemples de surfaces 
de translation ayant un groupe de Veech infiniment engendr\'e et de premi\`ere esp\`ece.

Notre preuve repose sur une nouvelle approche de la notion
de feuilletage p\'eriodique par les involutions lin\'eaires.
\end{abstract}
\maketitle

\renewcommand{\abstractname}{Abstract}
\begin{abstract}
This paper deals with Prym eigenforms which are introduced previously by McMullen.
We prove several results on the directional flow on those surfaces, related to complete periodicity 
(introduced by Calta). More precisely we show that any homological direction is algebraically periodic, and
any direction of a regular closed geodesic is a completely periodic direction.   As a consequence we draw that 
the limit set of the Veech group of every Prym eigenform in some Prym loci of genus 
$3,4$, and $5$  is either empty, one point, or the full circle at infinity. We also construct new 
examples of translation surfaces satisfying the topological Veech dichotomy. As a corollary we obtain 
new translation surfaces whose Veech group is infinitely generated and of the first kind.
\end{abstract}

\section{Introduction}

\subsection{Periodicity and Algebraic Periodicity}

In his 1989 seminal work~\cite{Veech1989}, Veech provided first instances of translation surfaces 
whose the directional flows satisfy a remarkable property: for a given direction, the flow is either uniquely ergodic, or 
completely periodic. This property is subsequently called the {\em Veech dichotomy}.
Since then numerous efforts have been made in the study of the linear flows on translation surfaces,
to name a very few:~\cite{Mc1,Smillie2006,Moller2006,Smillie2010,Eskin:preprint}. \medskip

This paper deals with completely periodic linear flows. This aspect has been 
initiated in~\cite{Calta04}, and then developed later in~\cite{CaSm}. 
A useful invariant to detect completely periodic flows, introduced in Arnoux's thesis~\cite{Ar1}, 
is the Sah-Arnoux-Fathi (SAF) invariant. It is well known that the linear flow $\mathcal{F}_\theta$ in a direction 
$\theta\in \R\P^1$ on a translation surface $(X,\omega)$ (equipped with a transversal interval $I$)
provides an interval exchange transformation $T_\theta$, which is the first return map to $I$. 
The invariant of the flow in direction $\theta$ can be informally defined by
$$
SAF(T_\theta) = \int_I 1 \otimes (T_\theta(x)-x)dx \in \R \wedge_\Q \R,
$$
(the integral is actually a finite sum). If $\mathcal F_\theta$ is periodic, that is when every leaf of $\mathcal{F}_\theta$ is 
either a closed curve, or an interval joining two zeros of $\omega$, then $SAF(T_\theta)=0$. However 
the converse is not true in general. 
Following this remark, the direction $\theta$ will be called {\em algebraically periodic} if the 
SAF-invariant of the flow $\mathcal F_\theta$ vanishes. \medskip

A translation surface $(X,\omega)$ is {\em completely periodic} if for every $\theta\in \R\P^1$
for which $\mathcal F_\theta$ has a closed regular orbit, the flow $\mathcal F_\theta$ is completely periodic.
We have the corresponding ``algebraic'' definition: the surface $(X,\omega)$ is {\em completely algebraically periodic} 
if the SAF-invariant of $\mathcal F_\theta$ vanishes in any homological direction 
($\theta\in \R\P^1$ is homological if it is the direction of a vector $\int_c \omega\in \C\simeq \R^2$ for 
some  $c \in H_1(X,\Z)$). These notions are introduced in \cite{Calta04} and \cite{CaSm}.\medskip

Flat tori and their ramified coverings are both completely periodic and  completely algebraically periodic; in this case, up to a 
renormalization by $\GL^+(2,\R)$, the set of homological directions is $\Q\cup\{\infty\}$. In \cite{Calta04}, Calta proved that these two
properties also coincide for genus $2$ translation surfaces, in which case the set of homological directions is $K\P^1$, where $K$ is 
either $\Q$ or a real quadratic field over $\Q$. A surface in $\Omega \mathcal M(2)$ is completely periodic  if and only if it is a Veech 
surface, but there are completely periodic surfaces in $\Omega \mathcal M(1,1)$ that are not Veech surfaces 
(actually, most of them are not Veech surfaces). \medskip

We will say that a quadratic differential is algebraically periodicity (respectively, completely periodic 
in the sense of Calta) if its orientating double covering is. In genus two, there are two strata, $\Omega \mathcal M(2)$ 
and $\Omega \mathcal M(1,1)$, which are equivalent respectively to the strata $\QQQ(-1^5,1)$ and $\QQQ(-1^6,2)$ 
of quadratic differentials over $\C\mathbb{P}^1$. In this paper we prove the following

\begin{MainThm}\label{thm:A}
Let $(Y,q)$ be  quadratic differential in one of the following strata
$$
\QQQ(-1^4,4),\ \QQQ(-1^3,1,2),\ \QQQ(-1,2,3),\ \QQQ(8),\ \QQQ(-1,1,4).
$$
If $(Y,q)$ is completely algebraically periodic then it is completely periodic in the sense of Calta.
\end{MainThm}
The method developed in this paper also allows us to get new {\em primitive} examples of non-Veech surfaces satisfying 
the topological Veech dichotomy.
\begin{MainThm}\label{thm:B}
On any quadratic differential in the strata $\QQQ(8)$ or $\QQQ(-1,2,3)$ that is stabilized by a pseudo-Anosov
homeomorphism, the directional foliations $\mathcal F_\theta, \, \theta \in \R\P^1,$ satisfy the  topological Veech dichotomy, namely either ${\mathcal F}_\theta$ is either periodic, or minimal.
\end{MainThm}
Note that above theorem is false in general for other strata. The last result we show concerns limit set of Veech groups.
\begin{MainThm}\label{thm:C}
For any quadratic differential in the collection of strata $\QQQ(8)$, $\QQQ(-1,2,3)$ or $\QQQ(-1^3,1,2)$ 
the limit set of its Veech group is either the empty set, a single point, or the full circle at infinity.

In particular, there exist infinitely many quadratic differentials in those strata
whose Veech group is infinitely generated and of the first kind.
\end{MainThm}
First examples of surfaces satisfying  topological Veech dichotomy without being Veech surfaces have been 
constructed in~\cite{Smillie2007} (see also~\cite{HubSch04}). All known examples of such surfaces are  
ramified coverings of ``true'' Veech surfaces. Thus Theorem~\ref{thm:B} provides us with first examples 
which do not arise from this construction. 

\subsection{Prym loci and Prym eigenforms}

From the work of McMullen~\cite{Mc07}, it turns out that all the completely periodic surfaces in genus two belong to the loci of {\em
eigenforms for real multiplication} (see below for the definitions). Later McMullen~\cite{Mc6} proves the existence of similar loci in genus
$3,4$ and $5$. These loci are of interest since they are closed $\GL^+(2,\R)$-invariant sub-manifolds in the moduli spaces of Abelian
differentials. We will call surfaces in these loci {\em Prym eigenforms}.  Let us now  briefly recall the definitions of those objects. \medskip

Let $X$ be a compact Riemann surface, and $\inv: X \rightarrow X$ be an holomorphic involution of $X$. We define
$$
\Prym(X,\inv)=(\Omega^-(X,\inv))^*/H_1(X,\Z)^-,
$$
\noindent where $\Omega^-(X,\inv)=\{\eta \in \Omega(X): \, \inv^*\eta=-\eta\}$.  $\Prym(X,\inv)$ will be called the {\em Prym variety} of $X$, it is a sub-Abelian variety of the Jacobian variety $\mathrm{Jac}(X):=\Omega(X)^*/H_1(X,\Z)$. \medskip

In this paper, for any integer vector $\kappa=(k_1,\dots,k_n)$ with $k_i\geq0$, we denote by $\Prym(\kappa) \subset \H(\kappa)$ 
the subset of pairs 
$(X,\omega)$ such that there exists an involution $\inv:  X \ra X$ satisfying  $\inv^*\omega=-\omega$, and $\dim_\C\Omega^-(X,\inv)=2$. For Riemann surfaces in genus two, one has $\Prym(2) \simeq \H(2)$ and $\Prym(1,1) \simeq \H(1,1)$
(the hyperelliptic involution is by definition a Prym involution, which is actually unique). See
Figure~\ref{fig:prototype} for an example. \medskip

Let $Y$ be the quotient of $X$ by the Prym involution (here $g(Y)=g(X)-2$) and $\pi$ the corresponding (possibly ramified) double covering from 
$X$ to $Y$. By push forward, there exists a meromorphic quadratic differential $q$ on 
$Y$ (with at most simple poles) so that $\pi^*q=\omega^2$. Let $\kappa'=(d_1,\dots,d_r)$ be 
the integer vector that records the orders of the zeros and poles of $q$. Then there is a  $\GL^+(2,\R)$-equivariant bijection
between $\QQQ(\kappa')$ and $\Prym(\kappa)$~\cite[p.~6]{Lanneau04}. \medskip

All the strata of quadratic differentials of dimension $5$ are recorded in Table~\ref{tab:strata:list}. It turns out that  the corresponding Prym varieties have complex  dimension two.  

\begin{table}[htb]
$$
\begin{array}{c|l|l|c}
&\text{Stratum of } Y    & \textrm{Prym loci} & \textrm{genus of } X  \\ 
\hline
1 & \QQQ(-1^6,2)  & \Prym(1,1) \simeq \H(1,1)  &  2 \\
2 & \QQQ(-1^2,6)  & \Prym(3,3) \simeq \H(1,1) &  4 \\
3 & \QQQ(1,1,2)  & \Prym(2,2,1,1) \simeq \H(0^{2},2)  &  4 \\
4 & \QQQ(-1^4,4)  & \Prym(2,2)^\mathrm{odd}  &  3 \\
5 & \QQQ(-1^3,1,2)  & \Prym(1,1,2)  &  3 \\
6 & \QQQ(-1,2,3)  & \Prym(1,1,4)  &  4 \\
7 &  \QQQ(8)  & \Prym(4,4)^\mathrm{even}  &  5 \\
8 & \QQQ(-1,1,4)  & \Prym(2,2,2)^\mathrm{even}  &  4 
\end{array}
$$
\caption{
\label{tab:strata:list}
Complete list of all strata of quadratic differentials having dimension $5$; The corresponding 
Prym variety has complex dimension $2$.
}
\end{table}
We now give the definition of Prym eigenforms. Recall that a quadratic order 
is a ring isomorphic to $\Ord_D = \Z[X]/(X^2+bX+c)$, where $D = b^2-4c>0$
(quadratic orders being classified by their discriminant $D$).
\begin{Definition}[Real multiplication]
\label{def:Real:mult}
Let $A$ be an Abelian variety of dimension $2$. We say that $A$ admits a real 
multiplication by $\Ord_D$ if there exists an injective homomorphism 
$\mathfrak{i}: \Ord_D \ra \mathrm{End}(A)$, such that $\mathfrak{i}(\Ord_D)$ 
is a self-adjoint, proper subring of $\mathrm{End}(A)$ ({\em i.e.} for any $f\in \mathrm{End}(A)$, if there exists 
$n\in \Z\backslash \{0\}$ such that $nf\in \mathfrak{i}(\Ord_D)$ then $f \in \mathfrak{i}(\Ord_D)$).
\end{Definition}

\begin{Definition}[Prym eigenform]
\label{def:Prym:eig:form}
For any quadratic discriminant $D>0$,  we denote by $\PrD(\kappa)$ the set of 
$(X,\omega) \in \Prym(\kappa)$ such that  $\dim_{\C}\Prym(X,\inv)=2$, $\Prym(X,\inv)$ admits a multiplication by $\Ord_D$, 
and $\omega$  is an eigenvector of $\Ord_D$. Surfaces in $\PrD(\kappa)$ are called {\em Prym eigenforms}.
\end{Definition}

Prym eigenforms do exist in each Prym locus described in Table~\ref{tab:strata:list}, as  real multiplications arise naturally with pseudo-Anosov homeomorphisms commuting  with  $\inv$ (see Theorem~\ref{theo:hyp:Prym:eig}).  

\subsection{Statements of results} The main purpose of this paper is to investigate the periodicity of 
linear flows on translation surfaces which are Prym eigenforms in the strata listed in Table~\ref{tab:strata:list}.  Our first results strengthen 
results of Calta and McMullen in genus two (that is to say cases $(1)-(2)-(3)$ of Table~\ref{tab:strata:list}).

\begin{Theorem}
\label{thm:Alg:Periodic}
Any Prym eigenform in the loci $\Prym(\kappa)$ given by the cases $(4)-(5)-(6)-(7)-(8)$ of Table~\ref{tab:strata:list} is completely algebraically periodic.
\end{Theorem}

Conversely, we have

\begin{Theorem}
\label{thm:CAP:eigen}
Let $(X,\omega)$ be a translation surface in one of the Prym loci $\Prym(\kappa)$ given by cases $(4)-(5)-(6)-(7)-(8)$ of Table~\ref{tab:strata:list}. Assume that $(X,\omega)$ is completely algebraically periodic, and the set of homological directions of $(X,\omega)$ is $K\P^1$, where $K$ is either  $\Q$, or a real quadratic field. Then the surface $(X,\omega)$ is a Prym eigenform.
\end{Theorem} 

To prove Theorem~\ref{thm:CAP:eigen}, we need the following theorem which relates complete algebraic periodicity and complete periodicity in Prym loci.

\begin{Theorem}\label{thm:CAP:to:CP}
Let $(X,\omega)$ be a translation surface in one of the Prym loci given by the cases  $(4)-(5)-(6)-(7)-(8)$ of Table~\ref{tab:strata:list}. If $(X,\omega)$ is
completely algebraically periodic, then it is completely periodic in the sense of Calta.
\end{Theorem}

As a consequence of Theorems~\ref{thm:Alg:Periodic} and~\ref{thm:CAP:to:CP}, we draw
\begin{Corollary}
\label{cor:eig:form:CP}
Every Prym eigenforms in the loci shown in Table~\ref{tab:strata:list} is completely periodic in the sense of Calta.
\end{Corollary}

The following theorem involves only Prym eigenforms in the loci given by cases $(6)-(7)$ of Table~\ref{tab:strata:list}.

\begin{Theorem}
\label{theo:cp:stronger}
Let $(X,\omega)$ be  a Prym eigenform in $\Omega  E_D(4,4)^{\rm even}$ or $\Omega
E_D(1,1,4)$ having all the periods (relative and absolutes) in $K(\imath)$, where $K=\Q(\sqrt{D})$. 
Then for any direction $\theta \in K\P^1$ the flow $\mathcal F_\theta$ is completely periodic. In 
particular, $(X,\omega)$ satisfies the  topological Veech dichotomy.
\end{Theorem}

\begin{Remark}
Theorem~\ref{theo:cp:stronger} is not true for every Prym loci. For instance, every completely algebraically periodic surface $(X,\omega)\in \Omega  E_D(1,1)$ with relative and absolute periods in $K(\imath)$, which is not a Veech surface,  admits an irrational splitting into two isogenous tori~\cite{CM2006}. In particular, the topological Veech dichotomy fails for such surfaces.
\end{Remark}

We will show (Section~\ref{sec:limit:sets}) that if $(X,\omega)$ is stabilized by an affine pseudo-Anosov homeomorphism, then the set of 
directions $\theta \in K\P^1$ that are  fixed by parabolic elements in the Veech group is dense in $\R\P^1$. 
As a corollary, we get 

\begin{Theorem}
\label{theo:infinite:limitset}
Let $(X,\omega)$ be a Prym form in $\Prym(4,4)^{\rm even}\sqcup\Prym(1,1,4)\sqcup\Prym(1,1,2)$. Then
the limit set of the Veech group $\SL(X,\omega)$ is either the empty set, a single point, or the full circle at infinity.

Moreover there exist infinitely many surfaces in 
$\Prym(4,4)^{\rm even}\sqcup\Prym(1,1,4)\sqcup\Prym(1,1,2)$ 
whose Veech group is infinitely generated and of the first kind.
\end{Theorem}

\subsection{Outline of the paper}

We conclude by sketching the proof of our results. It involves
the dynamics of interval exchange transformations (and linear involutions), 
the SAF-invariant and the kernel foliation in Prym loci.

\begin{enumerate}
\item   To prove Theorem~\ref{thm:Alg:Periodic} we use an invariant introduced by McMullen similar to the SAF-invariant: the Galois flux. 
Let $T$ be an interval exchange transformation (IET), and let $\lbd_\alpha, t_\alpha, \, \alpha \in \Alp,$ be respectively the lengths of the exchanged intervals and their translation lengths. The Galois flux of $T$ is defined {\em only} if  the translation lengths $t_\alpha$  lie in a real quadratic field $K \subset \R$, namely
$$
{\rm flux}(T)=\sum_{\alpha\in \Alp} \lbd_\alpha t'_\alpha, \textrm{ where $t'_\alpha$ is the Galois conjugate of $t_\alpha$}.
$$
\noindent It turns out (see Theorem~\ref{thm:zero:flux}) that if $(X,\omega)\in\PrD(\kappa)$ 
and having all absolute periods in $K(\imath)$ then for any $\theta \in K\P^1, {\rm flux}(T_\theta) = 0$, where $T_\theta$ is the IET defined by the first return map of the flow in direction $\theta$ to a transversal interval in $X$. The two invariants are related by Proposition~\ref{prop:flux:n:SAF}. Namely, under the 
additional assumption: if the relative periods of $\omega$ are also in $K(\imath)$ then  
$\fl(T_\theta)=0$ implies  $SAF(T_\theta)=0$.

Now if the relative periods of $\omega$ are not in $K(\imath)$ then we can
``perturb'' $\omega$ in $\PrD(\kappa)$ to get a new form $\omega'$ (by using the kernel foliation, see Section~\ref{sec:inv:kernel}) 
so that the relative periods of  $\omega'$ belong to $K(\imath)$. Thus by the preceding discussion $SAF(T'_\theta)=0$ ( $T'_\theta$ is the IET defined by flow in direction $\theta$ on $X'$).We then conclude with Proposition~\ref{prop:SAF:unchanged}: If the ``perturbation'' is small enough then
$$
SAF(T'_\theta)=SAF(T_\theta).
$$
\noindent This proves Theorem~\ref{thm:Alg:Periodic}. \medskip

\item It is well known that linear flows on translation surfaces are encoded by interval exchange transformations.
Since we will work with non-orientable measured foliations defined by quadratic differentials, 
it will be more convenient to use the coding provided by {\em linear involutions}.
By~\cite{Lanneau:rauzy} one can still define a ``first return'' of the non-orientable
foliation (that is no longer a flow) to a transverse interval, which gives a linear involution 
defined over $d$ intervals (see Section~\ref{sec:LI}). It turns out that the number $d$ of exchanged 
intervals is related to the dimension of the Prym locus, namely $d=\dim_\C \Prym(\kappa)+1$. 

Obviously complete periodicity for a foliation or for its associated linear involution is the same. In  view  of this,  we  will  deduce  Theorem~\ref{thm:CAP:to:CP}
from results on linear involutions (see Section~\ref{subsec:cp}).
We  briefly  sketch a proof here. Let $(X,\omega)\in \PrD(\kappa)$ be a Prym form which has 
a vertical cylinder (we normalize so that any homological direction belongs to a quadratic field). We will 
consider the cross section $T$ of the vertical {\it foliation} to some full transversal interval
on the quotient $X/<\inv>$. Let us resume the situation: $T$ is a linear involution defined over $d=6$ letters 
(for all the Prym loci in Table~\ref{tab:strata:list}, $\dim_\C\Prym(\kappa)=5$) having a periodic orbit, and $SAF(T)=0$. 
We want to show that $T$ is completely periodic. \medskip
 
\noindent If $T$ is defined over $2$ or $3$ intervals then the proof is immediate. We prove the assertion for $d=6$ by 
induction on the number of intervals, we pass from $d$ intervals to $d-1$ intervals by applying 
the Rauzy induction (which preserves the SAF-invariant, see Section~\ref{subsec:cp}). \medskip

\item Theorem~\ref{theo:cp:stronger} is a refinement of Theorem~\ref{thm:CAP:to:CP} 
by inspecting the possible degenerations of linear involutions (see Section~\ref{sec:d=6}). \medskip

\item For the proof of Theorem~\ref{theo:infinite:limitset}, we first remark that if the limit set of the Veech group $\SL(X,\omega)$ of $(X,\omega)$ contains at least two points, then $(X,\omega)$ is stabilized by an affine pseudo-Anosov homeomorphism $\phi$. It follows that all the relative and absolute periods of $\omega$ belong to $K(\imath)$, where $K=\Q(\mathrm{Tr}(D\phi))$ (see \cite{Mc2} Theorem~9.4).  To show that the limit set of  $\SL(X,\omega)$ is the full circle at infinity, it is sufficient to show that the set of directions which are fixed by  parabolic elements of $\SL(X,\omega)$ is dense in $\R\P^1$. For the cases of $\Prym(1,1,4)$ and $\Prym(4,4)^{\rm even}$, this follows from  Theorem~\ref{theo:cp:stronger} together with a criterion for a periodic direction to be parabolic (that is to be fixed by a parabolic element in $\SL(X,\omega)$, see Proposition~\ref{prop:cp:unstab:decomp}). For the case $\Prym(1,1,2)$, this follows from a similar result to Theorem~\ref{theo:cp:stronger} (see Corollary~\ref{cor:reformulation:d=6}), and a careful inspection of topological models for cylinders decompositions of surfaces in $\Prym(1,1,2)$ (see Section~\ref{sec:limit:sets}). \medskip

\item We can now give the proof of the main theorems. Theorem~\ref{thm:A} and Theorem~\ref{thm:C} are  reformulations of Theorem~\ref{thm:CAP:to:CP} and Theorem~\ref{theo:infinite:limitset} respectively. For Theorem~\ref{thm:B}, we first use Theorem~\ref{theo:hyp:Prym:eig} to derive that the orienting double covering of a quadratic differential in $\QQQ(-1,2,3)\sqcup\QQQ(8)$ which is stabilized by an affine pseudo-Anosov homeomorphism is a Prym eigenform for some quadratic order $\mathcal{O}_D$. Moreover, by Theorem~9.4 in \cite{Mc2}, we know that all the periods of this Abelian differential lie in $K(\imath)$, where $K=\Q(\sqrt{D})$. Thus Theorem~\ref{thm:B} follows from Theorem~\ref{theo:cp:stronger}.
\end{enumerate}

\subsection*{Acknowledgments}

We would like to thank Corentin Boissy, Pascal Hubert, and Pierre Arnoux for useful discussions.  We would also 
like to thank the  Centre de Physique  Th\'eorique in Marseille, 
the Universit\'e Bordeaux 1 and Institut Fourier in Grenoble for the hospitality during 
the preparation of this work.
Some of the research visits which made this collaboration possible were supported by 
the ANR Project GeoDyM. The authors are partially supported by the ANR Project GeoDyM.
We also thank~\cite{sage} for computational help.

\section{Interval exchange, Sah-Arnoux-Fathi invariant and McMullen's flux}
\label{sect:IET:SAF:CAP}

In this section we recall necessary background on interval exchange transformations
and we will make clear the relations between the SAF-invariant introduced by Arnoux in his thesis~\cite{Ar1},
the $J$-invariant introduced by Kenyon-Smillie~\cite{KS} and Calta~\cite{Calta04}, and the 
flux introduced by McMullen~\cite{Mc1}.

\subsection{Interval exchange transformation and SAF-invariant}
\label{sect:IET:SAF}

An interval exchange transformation (IET) is a map $T$ from an interval $I$ into itself defined 
as follows: we divide $I$ into finitely many subintervals of the form $[a,b)$. On each of 
such interval, the restriction of $T$ is a translation: $x \mapsto x+t$. By convention, the map $T$ is continuous from the right at the endpoints of the subintervals.  Any IET can
be encoded by a combinatorial data $(\Alp,\pi)$, where $\Alp$ is a finite alphabet,
$\pi=(\pi_0,\pi_1)$ is a pair of one-to-one maps   $\pi_\eps: \Alp \rightarrow \{1,\dots,d\}, \; d=|\Alp|$, together with a vector in the positive cone
$\lbd=(\lbd_\alpha)_{\alpha \in \Alp} \in \R_{>0}^{|\Alp|}$. The permutations  $(\pi_0,\pi_1)$ encodes how the intervals are exchanged,  and the vector $\lbd$ encodes the lengths of the intervals.
Following Marmi, Moussa, Yoccoz~\cite{Marmi:Moussa:Yoccoz}, we denote these intervals by 
$\{I_\alpha, \ \alpha\in \Alp\}$, the length of the interval $I_\alpha$ is $\lbd_\alpha$. 
Hence the restriction of $T$ to $I_\alpha$ is $T(x)=x+t_\alpha$ for some translation length $t_\alpha$.
Observe that $t_\alpha$ is uniquely determined by $\pi$ and $\lbd$. \medskip

A useful tool to detect periodic IET is given by the Sah-Arnoux-Fathi invariant
(SAF-invariant). It is defined by (see~\cite{Ar1}):
$$
SAF(T)=\sum_{\alpha\in \Alp} \lbd_\alpha \wedge_\Q t_\alpha.
$$
It turns out that if $T$ is periodic then $SAF(T)=0$. However the  converse is not true in general.

\subsection{$J$-invariant, SAF, and algebraic periodic direction}
\label{subsec:J-inv}
Let $(X,\omega)$ be a translation surface.
If $\mathbf{P}$ is a polygon in $\R^2$ with vertices $v_1,\dots,v_n$ which are numbered in 
counterclockwise order about the boundary of $P$, then the
$J$-invariant of $\mathbf{P}$ is $J(\mathbf{P})=\sum_{i=1}^n v_i \wedge v_{i+1}$ (with
the dummy condition $v_{n+1}=v_1$). Here $\wedge$ is taken to mean
$\wedge_\Q$ and $\R^2$ is viewed as a $\Q$-vector space. $J(\mathbf{P})$ is a
translation invariant ({\em e.g.} $J({\bf P}+\overrightarrow{v})=J({\bf P})$), thus this
permits to define $J(X,\omega)$ by $\sum_{i=1}^k J({\bf P}_i)$ where ${\bf P}_1 \cup \dots
\cup {\bf P}_k$ is a cellular decomposition of $(X,\omega)$ into planar
polygons (see~\cite{KS}).

The SAF-invariant of an interval exchange is related to the $J$-invariant as 
follows. We define a linear projection $J_{xx}:\R^2
\wedge_\Q \R^2 \rightarrow \R \wedge_\Q \R$ by
$$
J_{xx} \left( \left( \begin{smallmatrix} a \\ b\end{smallmatrix}\right)
  \wedge  \left( \begin{smallmatrix} c \\ d\end{smallmatrix}\right)
    \right) = a \wedge c.
$$
If $T$ is an interval exchange transformation induced by the first
return map of the vertical foliation on $(X,\omega)$ (on a transverse
interval $I$) then  $SAF(T) = J_{xx}(X,\omega)$. Note that the 
definition does not depend of the choice of $I$ if the interval meets
every vertical leaf (see~\cite{Ar1}). Hence this allows us to define
$$
SAF(X,\omega) = SAF(T),
$$
and we will say that $SAF(X,\omega)$ is the SAF-invariant of $(X,\omega)$ in 
the {\it vertical direction}. \medskip

Following Calta~\cite{Calta04}, one also defines the SAF-invariant of $(X,\omega)$ in any 
direction $k \in \R\P^1$ ($k \neq \infty = \left(\begin{smallmatrix} 0 \\ 1 \end{smallmatrix}\right)$)
as follows. Let $g\in GL^+(2,\R)$ be a matrix that sends the vector $\left(\begin{smallmatrix} 1 \\ k \end{smallmatrix}\right)$
to the vector $\left(\begin{smallmatrix} 0 \\ 1 \end{smallmatrix}\right)$. Then we define 
the SAF-invariant of $(X,\omega)$ {\it in direction $k$} to be $J_{xx}(g\cdot(X,\omega))$.

\subsection{Galois flux}

For the remaining of this section  $K$ will be a real quadratic field. There is a unique positive square-free integer $f$ such that $K=\Q(\sqrt{f})$. The {\em Galois conjugation} of $K$ is given by $u+v\sqrt{f} \mapsto u-v\sqrt{f}, \; u,v \in \Q$. For any $x\in K$, 
we denote by $x'$ its Galois conjugate. An interval exchange transformation $T$ is {\em defined over } $K$ if its translation lengths $t_\alpha$ are all in $K$. In \cite{Mc2}, McMullen defines the {\em Galois flux} of an IET $T$ defined over $K$ to be
$$
{\rm flux}(T)=\sum_{\alpha\in \Alp} \lbd_\alpha t'_\alpha.
$$
Observe that for all $n\in \N$, $\fl(T^n)=n\fl(T)$. In particular $\fl(T)=0$ if $T$ is periodic.
The flux is closely related to the SAF-invariant as we will see. 

\subsection{Flux of a measured foliation}

Let $(X,\omega)$ be a translation surface. The real form $\vr={\rm Re}(\omega)$ defines a measured foliation $\mathcal{F}_\vr$ on 
$X$: the leaves of $\mathcal{F}_\vr$ are vertical geodesics of the flat metric defined by $\omega$. 
For any interval $I$, transverse to $\mathcal{F}_\vr$, the cross section of the flow is an IET. We say that $I$ is 
{\em full transversal} if it intersects all the leaves of $\mathcal{F}_\vr$. If all the absolute periods of 
$\vr$ belong to the field $K$, that is $[\vr] \in H^1(X,K) \subset H^1(X,\R)$, then the first return map to 
$I$ is defined over $K$, and we have

\begin{Theorem}[McMullen~\cite{Mc2}]
\label{thm:fol:flux}
Let $T$ be the first return map of $\mathcal{F}_\vr$ to a full transversal interval. 
If $[\vr] \in H^1(X,K)$, then we have
$$
\fl(T)=-\int_X \vr\wedge \vr',
$$

\noindent where $\vr'\in H^1(X,K)$ is defined by $\vr'(c)=(\vr(c))'$, for all $c \in H_1(X,\Z)$. In particular, the flux 
is the same for any full transversal interval. In this case, we will call the quantity $-\int_X\vr\wedge \vr'$ the 
{\em flux} of the measured foliation $\mathcal{F}_\vr$, or simply the flux of $\vr$, and denote it by $\fl(\vr)$.
\end{Theorem}

\subsection{Complex flux} 
Let $K(\imath)$ be the extension of $K$ by $\imath =\sqrt{-1}$. Elements of $K(\imath)$ have the form 
$k=k_1+\imath k_2$, $k_1,k_2\in K$. We define $(k_1+\imath k_2)'=k'_1+\imath k'_2$, and 
$\ol{k_1+\imath k_2}=k_1-\imath k_2$. Suppose that $\omega \in \Omega(X)$ satisfies $[\omega] \in H^1(X,K(\imath))$ and
$$
\int_X\omega\wedge \omega'=0,
$$
($\omega'$ is an element of $H^1(X,K(\imath))$). The {\em complex flux} of $\omega$ is defined by
$$
\Fl(\omega)=-\int_X \omega\wedge\ol{\omega}'.
$$
Note that we always assume that $\int_X\omega\wedge\omega'=0$ when we 
consider $\Fl(\omega)$. This condition holds, for example, if $[\omega']$ is represented by a holomorphic 1-form.

In the following proposition, we collect  the important properties of the complex flux: 

\begin{Proposition}[McMullen~\cite{Mc2}]\label{prop:flux:property} \hfill
\begin{itemize}
 \item[a)] For any $k\in K(\imath)$, $\Fl(k\omega)=k\ol{k}'\Fl(\omega)$.

 \item[b)] If $\vr={\rm Re}(\omega)$, then
$$
\fl(\vr)=-\frac{1}{4}\int_X(\omega+\ol{\omega})\wedge (\omega'+\ol{\omega}')=\frac{1}{2}{\rm Re}(\Fl(\omega)).
$$
(here we used the condition $\int_X\omega\wedge \omega'=0$).

\item[c)] Assume that $\Fl(\omega)=0$. Let $k=k_2/k_1 \in  K \P^1, \; k_1,k_2 \in K$, and 
$\vr={\rm Re}((k_1+\imath k_2)\omega)$. Then $\mathcal{F}_\vr$ is the foliation by geodesic of slope 
$k$ in $(X,\omega)$, and we have
$$
\fl(\mathcal{F}_\vr)=\fl(\vr)=\frac{1}{2}{\rm Re}((k_1+\imath k_2)\omega)=\frac{1}{2}{\rm Re}((k_1+\imath k_2)(k'_1-\imath k'_2)\Fl(\omega))=0.
$$
\end{itemize}
\end{Proposition}

\subsection{Periodic foliation}
Given a cylinder  $C$ in $(X,\omega)$, we denote its width and height by $w(C)$ and $h(C)$ respectively.
If the vertical foliation is completely periodic, then $X$ is decomposed into cylinders in the this direction. 
It turns out that the imaginary part of $\Fl(\omega)$ provides us with important information on 
the cylinders. Namely the following is true:

\begin{Theorem}[McMullen~\cite{Mc2}]
\label{thm:Im:CFlux}
Assume that $[\omega] \in H^1(X,K(\imath))$, $\int \omega\wedge\omega'=0$, and the foliation $\mathcal{F}_\vr$ is periodic, where $\vr={\rm Re}(\omega)$. Let $\{C_j\}_{1\leq j\leq m}$ be the the vertical cylinders of $X$. Then we have
$$
\sum_{1\leq j \leq m} h(C_j)w(C_j)' =\frac{1}{2}{\rm Im}(\Fl(\omega)).
$$
\end{Theorem}

Recall that we have $N(k)=kk'\in \Q$, for any $k\in K$. For any cylinder $C$, 
the modulus of $C$ is defined by $\mu(C)=h(C)/w(C)$. A direct consequence of 
Theorem~\ref{thm:Im:CFlux} is the following useful corollary.

\begin{Corollary}
\label{cor:mod:relation}
If $\mathcal{F}_\vr$ is periodic, and the complex flux of $\omega$ vanishes, then the 
moduli of the vertical cylinders satisfy the following rational linear  relation
$$
\sum_{1 \leq j \leq m} \mu(C_j)N(w(C_j))=0.
$$
\end{Corollary}

\subsection{Prym eigenform and complex flux}

A  remarkable property of Prym loci is that the complex flux of 
a Prym eigenform (of a real quadratic order) vanishes. 

\begin{Theorem}[\cite{Mc2} Theorem $9.7$]
\label{thm:zero:flux}
Let $(X,\omega)$ be  a Prym eigenform  belonging to some locus $\PrD(\kappa)$. After   replacing  $(X,\omega)$   by  $A\cdot(X,\omega)$ for a suitable $A\in \GL^+(2,\R)$, we can assume that all the absolute periods of $\omega$ are in $K(\imath)$, where $K=\Q(\sqrt{D})$. We have
$$
\int_X \omega\wedge\omega'=0 \qquad \textrm{and} \qquad \Fl(\omega)=0.
$$
\end{Theorem}

\begin{proof}
Let $T$ be a generator of the order $\Ord_D$. We have a pair of $2$-dimensional 
eigenspaces $S\oplus S' = H^{1}(X,\mathbb R)^-$ on which $T$ acts with
eigenvalues $t,t'$ respectively. Since $T$ is self-adjoint, $S$ and $S'$ are
orthogonal with respect to the cup product.

The eigenspace $S$ is spanned by $\mathrm{Re}(\omega)$ and $\mathrm{Im}(\omega)$.
These forms lie in $H^{1}(X,K)$. The Galois conjugate of any form
$\alpha \in H^{1}(X,K) \cap S$ satisfies $T\alpha'=t' \alpha'$, and 
hence belongs to $S'$. In particular $\mathrm{Re}(\omega)'$ and $\mathrm{Im}(\omega)'$ are
orthogonal to $\mathrm{Re}(\omega)$ and $\mathrm{Im}(\omega)$. This shows
$$
\int_X \omega\wedge\omega'=0 \qquad \textrm{and} \qquad 
\Fl(\omega)=-\int_X\omega\wedge \overline{\omega}'=0.
$$
\end{proof}

\begin{Corollary}
\label{cor:eig:flux:nuls}
If $(X,\omega)$ is a Prym eigenform for a quadratic order $\Ord_D$ such that $[\omega]\in H^1(X,K(\imath))$, 
where $K=\Q(\sqrt{D})$, then for any $k \in K\P^1$, the  flux of the foliation by geodesics in direction $k$ vanishes.
\end{Corollary}

\subsection{Relation between SAF-invariant and complex flux}

\begin{Proposition}
\label{prop:flux:n:SAF}
Let $(X,\omega)$ be as in Theorem~\ref{thm:zero:flux}. Assume that all the relative periods of $\omega$ are also in $K(\imath)$. Then  $\fl(\omega)=0$ implies $SAF(\omega)=0$. Here, $\fl(\omega)$ and $SAF(\omega)$ denote the corresponding invariants of the vertical flow on $(X,\omega)$.
\end{Proposition}

\begin{proof}
Let $I$ be a full transversal interval for the vertical flow, and $T$ be the IET induced by the first return map on $I$. We denote the lengths of the permuted intervals $I_\alpha$ by $\lbd_\alpha$ and the 
translation lengths by $t_\alpha$ so that $T(x)=x+t_\alpha$ for any $x\in I_\alpha$. The assumption on relative periods implies that $\forall \alpha\in \Alp$, $\lbd_\alpha \in K$. 
Since $K=\Q(\sqrt{f})$, then we can write $t_\alpha=x_\alpha+y_\alpha\sqrt{f}$ with $x_\alpha,y_\alpha\in \Q$. Then $t'_\alpha=x_\alpha-y_\alpha\sqrt{f}$. The condition ${\rm flux}(T)=0$ is then equivalent to
$$
\sum_{\alpha\in\Alp}\lbd_\alpha x_\alpha=\sqrt{f}\sum_{\alpha\in\Alp}\lbd_\alpha y_\alpha.
$$

\noindent Since $\lbd_\alpha\in K$, it follows
$$
\sum_{\alpha\in\Alp}\lbd_\alpha y_\alpha=A+B\sqrt{f} \quad \text{ and } \quad \sum_{\alpha\in\Alp}\lbd_\alpha x_\alpha=fB+A\sqrt{f} \text{ with } A,B \in \Q.
$$

\noindent By definition of the $SAF$-invariant, we have
$$
SAF(T)=(\sum_{\alpha\in\Alp}\lbd_\alpha x_\alpha)\wedge_\Q1+ 
(\sum_{\alpha\in\Alp}\lbd_\alpha y_\alpha)\wedge_\Q \sqrt{f} =A\sqrt{f}\wedge_\Q 1+ A\wedge_\Q\sqrt{f}=0.
$$
\end{proof}

\begin{Corollary}\label{cor:eigen:rel:CAP}
Let $(X,\omega)$ be a Prym eigenform in $\Omega E_D(\kappa)$. Assume that all the periods of $\omega$ belong to $K(\imath)$, where $K=\Q(\sqrt{D})$. Then $(X,\omega)$ is completely algebraically periodic.
\end{Corollary}
\begin{proof}
 We first remark that the set of homological directions of $(X,\omega)$ is $K\P^1$. For any direction $\theta \in K\P^1$, there exists a matrix $g_\theta \in \GL^+(2,K)$ that maps $\theta$ to the vertical direction. Note that all the periods of $g_\theta\cdot\omega$ are in $K(\imath)$. From the properties of $\fl$, we know that $\fl(g_\theta\cdot\omega)=0$, thus $SAF(g_\theta\cdot\omega)=0$, which implies that the $SAF$-invariant vanishes in direction $\theta$.
\end{proof}

\section{Invariance of SAF along kernel foliation leaves}
\label{sec:inv:kernel}

\subsection{Kernel foliation}

Here we briefly recall the kernel foliation for Prym loci 
(see~\cite{EMZ,Masur:Zorich,Calta04,Lanneau:Mahn:composantes} for  related constructions). 
The kernel  foliation was introduced by Eskin-Masur-Zorich, and was certainly known to Kontsevich.

Let $(X,\omega)$ be a translation surface having several distinct zeros.
Roughly speaking a leaf of the kernel foliation through $(X,\omega)$ is the set 
of surfaces $(X',\omega')$ that belong to the same stratum and share the same  absolute periods as $(X,\omega)$, 
{\em i.e.} for any $c \in H_1(S,\Z)$, $\omega(c)=\omega'(c)$, where $S$ is the base topological surface homeomorphic to 
both $X$ and $X'$. The kernel foliation exists naturally in all strata of the moduli spaces of translation surfaces. 

Choose $\eps>0$ small enough so that, for every zero $P$ of $\omega$, the set
$D(P,\eps)=\{x\in X, \dist(P,x) < \eps\}$ is an embedded disc in $X$. If $P$ is a zero of order $k$, then 
$D(P,\eps)$ can be constructed from $2(k+1)$ half-discs as described in the left part of
Figure~\ref{fig:move:zero}.  Pick a vector $w \in \C$, $0< |w| < \eps$, and cut $D(P,\eps)$ along the rays in direction $\pm w$, we get $2(k+1)$ half-discs which are glued together such that all the centers are identified with $P$. We  modify the metric structure of $D(P,\eps)$ as follows: in the diameter of each half-disc, there is a unique point $P'$ such that $\overrightarrow{PP'}=w$, we can glue the half-discs in such a way that all the points $P'$ are identified. Let us denote by $D'$ the  domain obtained from this gluing. We can glue $D'$ to $X\setminus D(P,\eps)$ along $\partial D'$, which is the same as $\partial D(P,\eps)$. We  then get a translation surface $(X',\omega')$ which has the same absolute periods as $(X,\omega)$, and satisfies the following condition: if $c$ is a path in $X$  joining another zero of $\omega$ to $P$, and $c'$ is the corresponding  path in $X'$, then we have $\omega(c')=\omega(c)+w$. In this situation, we will say that $P$ is moved by $w$. By definition $(X',\omega')$ lies in the kernel foliation leaf through $(X,\omega)$.

\medskip

We also have a kernel foliation in Prym loci in Table~\ref{tab:strata:list}. Let  $(X,\omega)$ be a translation surface in a Prym locus given in Table~\ref{tab:strata:list}. Remark that the Prym involution $\inv$ always exchanges two zeros of $\omega$, say $P_1,P_2$.  Given $\eps$ and $w$ as above, to get a surface $(X',\omega')$ in the same Prym locus, it suffices to move $P_1$ by $w/2$ and move $P_2$ by $-w/2$. Indeed, by assumptions, the Prym involution exchanges $D(P_1,\eps)$ and $D(P_2,\eps)$. Let $D'_1$ and $D'_2$ denote the new domains we obtain from $D(P_1,\eps)$ and $D(P_2,\eps)$ after modifying the metric. It is easy to check that $D'_1$ and $D'_2$ are symmetric, thus the involution in $X\setminus(D(P_1,\eps)\sqcup D(P_2,\eps))$ can be extended to $D'_1 \sqcup D'_2$. Therefore we have an involution $\inv'$  on $X'$ such that ${\inv'}^*\omega'=-\omega'$, which implies that $(X',\omega')$ also belongs to the same Prym locus as $(X,\omega)$. 

In what follows we will denote the surface $(X',\omega')$ obtained from this construction by $(X,\omega)+w$. Let $c$ be a path on $X$ joining two zeros of $\omega$, and $c'$ be the corresponding path in $X'$. Then we have

\begin{itemize}
 \item if  two endpoints of $c$ are exchanged by $\inv$ then $\omega'(c')-\omega(c)=\pm w$,
 
 \item if one endpoint of $c$ is fixed by $\inv$, but the other is not, then $\omega'(c')-\omega(c)=\pm w/2$.

\end{itemize}

\noindent The sign of the difference is determined by the orientation of $c$.

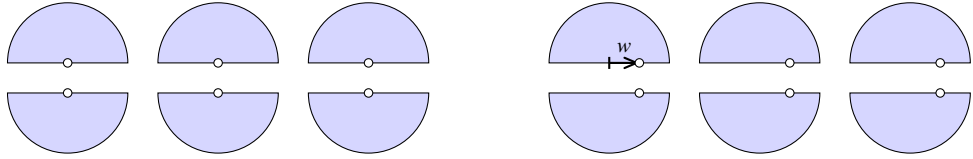
\begin{figure}[htbp]
\centering
\subfloat{
\begin{tikzpicture}[scale=0.4]
\foreach \x in {(-12,0), (-7,0), (-2,0), (6,0), (11,0), (16,0)} \filldraw[fill=blue!20!white!80] \x arc (0:180:2) -- cycle;
\foreach \x in {(-16,-1), (-11,-1), (-6,-1), (2,-1), (7,-1), (12,-1)} \filldraw[fill=blue!20!white!80] \x arc (180:360:2) -- cycle;
\draw[thick] (4,0) +(0,-0.2) -- +(0,0.2); \draw[->, >=angle 45, thick] (4,0) -- (5,0);
\draw (4.5,0) node[above] {$\scriptstyle w$};
\foreach \x in {(-14,0), (-14,-1), (-9,0), (-9,-1), (-4,0),(-4,-1), (5,0), (5,-1), (10,0), (10,-1), (15,0), (15,-1)} \filldraw[fill=white] \x circle (4pt);
\end{tikzpicture}
}
\caption{Moving a zero by a vector $w\in \R^2$}
\label{fig:move:zero}
\end{figure}

\subsection{Neighborhood of a Prym eigenform}

\begin{Lemma}
\label{lm:dim:prym:eig}
For any Prym locus $\Prym(\kappa)$ in Table~\ref{tab:strata:list}, and any discriminant 
$D \in \N, \; D \equiv 0,1 \mod 4$, if $\Omega E_D(\kappa) \neq \varnothing$, then $\dim_\C\Omega E_D(\kappa)=3$.
\end{Lemma}
\begin{proof}
Denote by  $\Sigma$ the set of zeros of $\omega$. Let $H^1(X,\C)^-$ and $H^1(X,\Sigma;\C)^-$ denote the eigenspace of $\inv$ with the eigenvalue $-1$ in $H^1(X,\C)$ and $H^1(X,\Sigma;\C)$ respectively. In a local chart which is given by a period mapping, a Prym form in $\Prym(\kappa)$ close to $(X,\omega)$ corresponds to a vector in $H^1(X,\Sigma;\C)^-$. Note that $\dim_\C H^1(X,\C)^-=4$ and $\dim_\C H^1(X,\Sigma;\C)^-=5$, and we have a natural surjective linear map $\rho: H^1(X,\Sigma;\C)^- \rightarrow H^1(X,\C)^-$.

Let $\hat{S}=\C\cdot[{\rm Re}(\omega)]\oplus \C\cdot[{\rm Im}(\omega)]\subset H^1(X,\C)^-$, where $[{\rm Re}(\omega)]$ and $[{\rm Im}(\omega)]$ are the cohomology classes in $H_1(X,\R)$ represented by ${\rm Re}(\omega)$ and ${\rm Im}(\omega)$. Since $\omega$ is an eigenform of a quadratic order ${\mathcal O}_D$, there exists an endomorphism $T$ of $H^1(X; \C)^-$ which generates ${\mathcal O}_D$ such that $\hat{S}$ is an eigenspace of $T$ for some real eigenvalue. A Prym eigenform in $\Omega E_D(\kappa)$ close to $(X,\omega)$ corresponds to a vector in $\rho^{-1}(\hat{S})$. Since $\dim_\C\hat{S}=2$ and $\dim_\C \ker\rho=1$, the lemma follows.
\end{proof}

\begin{Corollary}\label{cor:prym:eig:ker:fol}
For any $(X,\omega)\in \Omega E_D(\kappa)$,  if $(X',\omega')$ is a Prym eigenform in $\Omega E_D(\kappa)$ close enough to $(X,\omega)$, then there exists a unique pair $(g,w)$, where $g\in \GL^+(2,\R)$ close to $\Id$, and $w\in \R^2$ with $|w|$ small, such that $(X',\omega')=g\cdot(X,\omega)+w$.
\end{Corollary}
\begin{proof}
Let $(Y,\eta)=(X,\omega)+w$, with $|w|$ small, be a surface in the leaf of the kernel foliation through $(X,\omega)$. 
We denote by $[\omega]$ and $[\eta]$ the classes of $\omega$ and $\eta$ in $H^1(X,\Sigma;\C)^-$. Then we have 
$$
[\eta]-[\omega] \in \ker\rho,
$$
where $\rho: H^1(X,\Sigma;\C)^- \rightarrow H^1(X,\C)^-$ is the natural surjective linear map.
On the other hand, the action of $g\in \GL^+(2,\R)$ on $H^1(X,\Sigma;\C)^-$ satisfies 
$$
\rho(g\cdot[\tilde{\omega}])=g\cdot\rho([\tilde{\omega}]).
$$
\noindent Therefore the leaves of the kernel foliation and the orbits of $\GL^+(2,\R)$ are transversal. Since their dimensions are complementary, the corollary follows.
\end{proof}

\subsection{Kernel foliation and SAF-invariant}
In the remaining of  this section, $(X,\omega)$ is a translation surface in $\Prym(\kappa)$ where
$$
\kappa \in \left\{ (1,1), (3,3), (2,2)^{\rm odd}, (1,1,2), (4,4)^{\rm even}, (2,2,2)^{\rm even}, (1,1,4) \right\}.
$$

Our goal is to prove the following result on the invariance of the 
SAF-invariant by  moving along a leaf of the kernel foliation by a small vector.

\begin{Proposition}
\label{prop:SAF:unchanged}
For any $(X,\omega)\in\Prym(\kappa)$ there exists $\eps >0$ such that for any  $w \in \C$, with $|w|< \eps$
$$
SAF(X,\omega)=SAF((X,\omega)+w).
$$
\end{Proposition}

As a consequence, we draw our first theorem.

\begin{proof}[Proof of Theorem~\ref{thm:Alg:Periodic}]
We want to show that every Prym eigenform $(X,\omega)$ 
is completely algebraically periodic. Since $\omega$ is an eigenform for 
a real quadratic order $\Ord_D$, up to action of $\GL^+(2,\R)$ we can assume 
that all the absolute periods of $\omega$ are in $K(\imath)$, where $K=\Q(\sqrt{D})$. As a consequence, the set of homological directions 
of $(X,\omega)$ is $K\P^1$. \medskip

If $D$ is a square, then $K=\Q$, in which case, we can assume that all the absolute periods of $\omega$ belong to $\Z\oplus \imath\Z$. Thus $(X,\omega)$ is a ramified covering of the standard torus $\C/(\Z\oplus\imath\Z)$. It follows that for every direction $\theta\in \Q\cup\{\infty\}$, the linear flow in direction $\theta$ is periodic, which means that the SAF-invariant vanishes. Therefore, $(X,\omega)$ is completely algebraically periodic. \medskip

For the case where $K$ is a real quadratic field, given a direction $k\in K\P^1$, as usual we normalize so that $k$ is the vertical direction
$(0:1)$. Let $T$ be the first return map of the vertical flow to a full transversal interval.
All what we need to show is that $SAF(T)=0$. \medskip

Theorem~\ref{thm:zero:flux} ensures that $\fl(\omega)=0$. However, in view of applying
Proposition~\ref{prop:flux:n:SAF} we need  $\omega$ to have relative periods in $K(\imath)$, 
which is not necessary true. To bypass this difficulty we first apply Proposition~\ref{prop:SAF:unchanged}.
One remarks that all the relative periods of $\omega$ are determined by a chosen relative period and 
the absolute ones. Hence we can choose a small suitable vector $w\in \R^2$ such all the relative 
coordinates of $(X,\omega)+w$ are in $K(\imath)$ and Proposition~\ref{prop:SAF:unchanged} applies
{\em i.e.} $SAF(X,\omega) = SAF((X,\omega)+w)$. Since $(Y,\eta) = (X,\omega)+w$ is still an eigenform,
again Theorem~\ref{thm:zero:flux} gives $\fl(\eta)=0$. 
But now by Proposition~\ref{prop:flux:n:SAF}, we draw 
$SAF((Y,\eta))=0$. Hence the SAF-invariant of the vertical flow on $(X,\omega)$ also vanishes 
and Theorem~\ref{thm:Alg:Periodic} is proven.
\end{proof}

\subsection{Proof of Proposition~\ref{prop:SAF:unchanged}}

Recall that one wants to show that for any $(X,\omega)\in\Prym(\kappa)$ there exists 
$\eps >0$ such that for any  $w \in \C$, with~$|w|< \eps$
$$
SAF(X,\omega)=SAF((X,\omega)+w).
$$

We begin by choosing a full transversal interval $I=[P_0,P)$ 
for the vertical flow on $X$ with the following properties:

\begin{itemize}
\item $P_0$ is a zero of $\omega$, in case $\omega$ has three distinct zeros, we choose $P_0$ to be the zero which is fixed by $\inv$.
\item The rightmost endpoint $P$ of $I$ is not a zero of $\omega$.
\item The vertical leaf through $P$ intersects $I$ before reaching any zero of $\omega$ in both positive (upward) and negative (downward) directions.
\end{itemize}
Intervals satisfying those conditions do always exist (see~\cite{Lanneau:rauzy}). Let $T$ be the IET induced by the first 
return map of the vertical flow on $I$. Let $(\Alp,\pi)$ be the combinatorial data of $T$.  This map is described by a (maximal) partition 
$I_\alpha=[a_\alpha, b_\alpha)$, $\alpha \in \Alp$ of $I$, so that $T(x)=x+t_\alpha$ for $x\in I_\alpha$. The length of $I_\alpha$ is $\lbd_\alpha$.

For any $\alpha\in \mathcal A$, let  $C_\alpha$  be the oriented closed curve from a point $x_\alpha\in (a_\alpha,b_\alpha) \subset I_\alpha$ to $T(x_\alpha)$ along $I$, then back to $x_\alpha$ along a vertical leaf.   We have
$$
t_{\alpha} = T(x)-x = \int_{C_\alpha} \mathrm{Re}(\omega)=\vr(C_\alpha),
$$
\noindent where $\vr={\rm Re}(\omega)$. In particular one sees that $t_\alpha$ as the real part of some absolute period of $\omega$. \smallskip

Let $\{R_\alpha, \, \alpha \in \Alp\}$ be the rectangles associated to $I_\alpha$. By the hypothesis on $I$, for any $\alpha \in \Alp$, each vertical side of $R_\alpha$ contains exactly one zero of $\omega$ or the point $P$. We will denote by $\sigma_\alpha$ the segment in $R_\alpha$ (oriented from the left to the right) joining these two points. It follows
$$
\lbd_\alpha=\int_{\sigma_\alpha} \mathrm{Re}(\omega) = \vr(\sigma_\alpha).
$$

\begin{Lemma}
\label{lm:IET:stable}
There exists $\eps>0$ such that if $w\in \C, \, |w|<\eps$, then $(X,\omega)+w$ can be represented 
by a suspension of a new IET $T'$ defined on the same interval $I$, and having the same combinatorial
data $(\Alp,\pi)$ as $T$. 
\end{Lemma}

\begin{proof}
Recall that the surface $(X,\omega)+w$ is obtained by moving the zeros of 
$\omega$ that are exchanged by the Prym involution by $\pm w/2$. If $P_0$ is fixed, 
we keep the same transversal interval $I$. Otherwise $P_0$ is moved by $\pm w/2$, in this case we move $P$  by the same 
vector. The new transversal interval has the same length and direction as $I$. Thus if $|w|$ is small enough 
then the combinatorial data of the first return map are unchanged.
\end{proof}

We split the proof into two cases, depending whether $\omega$ has two or three zeros. \medskip

\subsubsection{Proof of Proposition~\ref{prop:SAF:unchanged}: case $\omega$ has two zeros}
Let $P_1$ denote the other zero of $\omega$. We define the subset $\Alp'$ of $\Alp$ as follows: 
$\alpha \in \Alp'$ if and only if $P_1$ is one of the endpoints of $\sigma_{\alpha}$, and the other endpoint is either $P_0$ or $P$. Let
$$
\begin{array}{ll}
\varepsilon : & \Alp' \rightarrow \{\pm 1\}, \\
\varepsilon(\alpha) = &  \left\{ \begin{array}{ll}
1 & \text{ if } P_1 \text{ is the right endpoint of } \sigma_\alpha,\\
-1 & \text{ if } P_1 \text{ is the left endpoint of } \sigma_\alpha.\\
\end{array} \right. 
\end{array}
$$
We choose a small vector $w\in \R^2$ given by Lemma~\ref{lm:IET:stable}. 

\begin{Lemma}
\label{lm:SAF:formula} Set $x={\rm Re}(w)$, we have
$$
SAF(T') - SAF(T) = x\wedge_{\Q}(\sum_{\alpha\in \Alp'} \varepsilon(\alpha)t_\alpha).
$$
\end{Lemma}

\begin{proof}
We let $\Alp'_+=\eps^{-1}(1)$ and $\Alp'_-=\eps^{-1}(-1)$. For any $\alpha \in \Alp$, let
$\sigma'_\alpha$ be the segment in $X'$ corresponding to $\sigma_\alpha$, it follows
$$
\left\{\begin{array}{ll}
\omega'(\sigma'_\alpha) = \omega(\sigma_\alpha) + w, & \text{for all } \alpha\in \Alp'_+, \\
\omega'(\sigma'_\alpha) = \omega(\sigma_\alpha) - w, & \text{for all } \alpha\in \Alp'_-.
\end{array} \right.
$$
Hence taking the real part of these complex numbers, we get
$$
\left\{ \begin{array}{ll}
\lbd'_\alpha=\lbd_\alpha &\text{for all } \alpha\in \Alp\setminus \Alp', \\
\lbd'_\alpha=\lbd_\alpha+x, & \text{for all } \alpha\in \Alp'_+, \\
\lbd'_\alpha=\lbd_\alpha-x, & \text{for all } \alpha\in \Alp'_-,
\end{array} \right.
$$
where $\lbd'_\alpha={\rm Re}(\omega'(\sigma'_\alpha))$. \medskip

From Lemma~\ref{lm:IET:stable}, we know that $(X',\omega')$ is a suspension of an IET $T'$ with the same 
combinatorial data as the ones of $T$. The lengths associated to $T'$ are by construction 
$(\lbd'_\alpha)_{\alpha\in\Alp}$. Since the surfaces in the same leaf of the kernel foliation have the 
same absolute periods, the translation lengths $(t_\alpha)_{\alpha\in\Alp}$ of $T$ and 
$(t'_\alpha)_{\alpha\in\Alp}$ of $T'$ are equal. This gives
\begin{multline*}
SAF(T')=\sum_{\alpha\in \Alp} \lbd'_\alpha\wedge_{\Q}t'_\alpha = 
SAF(T)+\sum_{\alpha\in \Alp'_+}x \wedge_{\Q}t_\alpha-\sum_{\alpha\in \Alp'_-}x\wedge_{\Q}t_\alpha = \\
= SAF(T)+x\wedge_{\Q}(\sum_{\alpha\in \Alp'} \varepsilon(\alpha)t_\alpha).
\end{multline*}
The lemma is proved.
\end{proof}

Since $t_\alpha=\rhoup(C_\alpha)$, for proving $SAF(T')=SAF(T)$ 
it is sufficient to show:

\begin{Lemma}
\label{lm:nul:homology:2zeros}
$$
\sum_{\alpha\in\Alp'} \varepsilon(\alpha)C_\alpha=0 \in H_1(X,\Z).
$$
\end{Lemma}

\begin{proof}
We first show the family $\{C_\alpha, \, \alpha\in \Alp'\}$ can be replaced by another family $\{\hat{C}_\alpha, \, \alpha\in \Alp'\}$, where $\hat{C}_\alpha$ is a simple closed curve homologous to $C_\alpha$ such that  all the curves $\hat{C}_\alpha$ have unique common point. To see this, let $R$ be a rectangle $I\times[0,\delta]$, with $\delta>0$ small enough so that $R$ can be embedded into $X$. Let $\hat{I}$ denote the upper side of $R$ (which is $=I\times\{\delta\}$), and $Q$ denote the center of $R$. For each $\alpha\in \Alp$, let $\hat{x}_\alpha$ denote the intersection of the vertical leaf through $x_\alpha$ with $\hat{I}$, and $s_\alpha$ denote the path in $R$ from $\hat{x}_\alpha$ to $T(x_\alpha)$ which is the union of the two segments joining $\hat{x}_\alpha$ and $T(x_\alpha)$ to $Q$. Finally, let $\hat{C}_\alpha$ be the simple closed curve which is the union of $s_\alpha$ and the vertical leaf from $T(x_\alpha)$ to $\hat{x}_\alpha$. By construction, $\hat{C}_\alpha$ is homologous to $C_\alpha$, and the family $\{\hat{C}_\alpha, \; \alpha\in \Alp'\}$ satisfies the condition stated above. \medskip

We now claim that $\displaystyle{\hat{X}= X\setminus (\cup_{\alpha\in \Alp'} \hat{C}_\alpha)}$ is not connected, and $(\cup_{\alpha\in \Alp'} \hat{C}_\alpha)$ is the oriented boundary of the component containing $P_1$. To see that $\hat{X}$ is not connected, we cut the rectangles $R_\alpha, \alpha \in \Alp$ along $\sigma_\alpha$, and reglue the pieces to get a polygon $\mathbf{P}$ in $\R^2$ whose vertices correspond
to the zeros $P_0,\ P_1$ and the rightmost endpoint $P$ of the interval $I$.

The set  $\bigcup_{\alpha\in \Alp'} \hat{C}_\alpha$ is then represented by a union of simple arcs joining a unique common point in the interior of ${\bf P}$ to  points in the boundary of $\mathbf{P}$. By definition, these arcs separate the set of vertices corresponding to $P_1$ from the set of vertices corresponding to $P_0$ and $P$, from which we deduce that $\hat{X}$ is disconnected. 

Now, for $\alpha \in \Alp'$, from the choice of the orientation of $\hat{C}_\alpha$, and the value of 
$\varepsilon(\alpha)$, we see that as one moves along $\varepsilon(\alpha)\hat{C}_\alpha$, 
the point $P_1$ is always located on the left. Thus we can conclude that $\displaystyle{\cup_{\alpha\in \Alp'} \varepsilon(\alpha) \hat{C}_\alpha}$ is the boundary of the component of $\hat{X}$ that contains $P_1$. Therefore 
$$\displaystyle{\sum_{\alpha\in \Alp'} \varepsilon(\alpha)  \hat{C}_\alpha =0\in H_1(X,\Z)}.$$
\noindent The lemma is then proven. 
\end{proof}

\subsubsection{Proof of Proposition~\ref{prop:SAF:unchanged}: case $\omega$ has three zeros}
The proof is similar to the previous case. Let $P_1, P_2$ be zeros permuted by $\inv$.
If $c$ is a segment from $P_0$ to $P_1$ then $\inv(c)$ is a segment from $P_0$ to $P_2$
and $\omega(\inv(c))=-\omega(c)$ (since $\inv^*\omega=-\omega$). We define the subset $\Alp'$ 
of $\Alp$ as follows: $\alpha \in \Alp'$ if and only if $\sigma_{\alpha}$ connects $P$ and one of the zeros $\{P_1,P_2\}$,
or two distinct zeros of $\omega$. As above, we define several subsets of $\Alp'$:
$$
\begin{array}{lcl}
\alpha \in \Alp'_{++} & \Leftrightarrow & \hbox{$\sigma_\alpha$ joins $P_1$ to $P_2$,}\\
\alpha \in \Alp'_{+}  & \Leftrightarrow & \hbox{$\sigma_\alpha$ joins $P_1$ to either $P_0$ or $P$, or $\sigma_\alpha$ joins either $P_0$ or $P$ to $P_2$},\\
\alpha \in \Alp'_{-}  & \Leftrightarrow & \hbox{$\sigma_\alpha$ joins either $P_0$ or $P$ to $P_1$, or $\sigma_\alpha$ joins $P_2$ to either $P_0$ or $P$},\\
\alpha \in \Alp'_{--} & \Leftrightarrow & \hbox{$\sigma_\alpha$ joins $P_2$ to $P_1$},\\
\end{array}
$$

\noindent and the associated function $\varepsilon: \Alp' \rightarrow \{\pm 1, \pm 2\}$ by
$$\varepsilon(\alpha)=\left\{
\begin{array}{rl}
2 & \text{ if } \alpha \in \Alp'_{++},\\
1 & \text{ if } \alpha \in \Alp'_{+},\\
-1 & \text{ if } \alpha \in \Alp'_{-},\\
-2 & \text{ if } \alpha \in \Alp'_{--}.\\
\end{array}\right.
$$

\noindent The reformulation of the definition of the kernel foliation, using above notations, gives

\begin{Lemma}\label{lm:kernel:SAF:3zeros}
$$
SAF(T') - SAF(T) = {\rm Re}(w)\wedge_{\Q}(\sum_{\alpha\in \Alp'} \varepsilon(\alpha)t_\alpha).
$$
\end{Lemma}

The proof of Lemma~\ref{lm:kernel:SAF:3zeros} is completely similar to the proof of Lemma~\ref{lm:SAF:formula}. Hence for proving Proposition~\ref{prop:SAF:unchanged}, we are left to show

\begin{Lemma}\label{lm:nul:hom:3zeros}
$$
\sum_{\alpha\in \Alp'}\varepsilon(\alpha)C_\alpha =0 \in H_1(X,\Z).
$$
\end{Lemma}

\begin{proof}
The argument follows the same lines as  the proof of Lemma~\ref{lm:nul:homology:2zeros}.
We can replace the family $\{C_\alpha, \alpha\in \Alp'\}$ by another family $\{\hat{C}_\alpha, \; \alpha \in \Alp'\}$ of simple closed curves, where $\hat{C}_\alpha$ is homologous to $C_\alpha$, such that the surface $\displaystyle{\hat{X}=X\setminus(\cup_{\alpha\in \Alp'}\hat{C}_\alpha)}$ is disconnected.  Let $X_1,X_2$ be the components of $X$ that contain $P_1$ and $P_2$ respectively, and $X_0$ be the complement of $X_1\cup X_2$ in $\hat{X}$. Observe that the curves $\{\hat{C}_\alpha, \, \alpha \in \Alp'\}$ separate $P_1$ from $P_2$, and $\{P_1,P_2\}$ from $\{P_0,P\}$. Thus we have $X_1\neq X_2$ and $\{P_0,P\}\subset X_0 \neq X_i, \; i=1,2$. Note that $X_0$ may be disconnected, in which case it has two components, one contains $P_0$, the other contains $P$. Set

$$C=\partial \overline{X}_1\cap \partial\overline{X}_2, \quad C_1=\partial \overline{X}_1\cap \partial\overline{X}_0, \quad C_2=\partial\overline{X}_2\cap \partial \overline{X}_0.$$

\noindent We choose the orientation of $C$ and $C_2$ to be the orientation induced from the orientation of $X_2$, and for $C_1$ the orientation induced from the orientation of $X_0$. Consequently we can write

$$\partial \overline{X}_0=C_1-C_2, \quad \partial \overline{X}_1=-C-C_1, \quad \partial \overline{X}_2=C+C_2.$$

\noindent Thus

$$C_1-C_2=-C-C_1=C+C_2=0 \in H_1(X,\Z).$$

\noindent Now, from the values of $\varepsilon(\alpha)$, and the orientation of $\hat{C}_\alpha$,  we have

$$
\sum_{\alpha\in\Alp'}\varepsilon(\alpha)\hat{C}_\alpha=2C+C_1+C_2=2(C+C_2)=0\in H_1(X,\Z).
$$
This concludes the proof of Proposition~\ref{prop:SAF:unchanged}.
\end{proof}
\section{Interval exchange transformations and linear involutions}
\label{sec:LI}

\subsection{Linear involutions}

The first return map of the  vertical flow on a translation surface $(X,\omega)$ to
an interval  $I$ defines an  interval exchange transformation (see Section~\ref{sect:IET:SAF}).  Such a
map is  encoded by  a partition  of $I$ into  $d$ subintervals that we
label by letters in some finite alphabet $\Alp$, and a permutation $\pi$ of $\Alp$. 
The length of these intervals is recorded by vector $\lbd$ with positive entries. 
The vector $\lbd$  is called the continuous datum  of $T$ and $\pi$
is called the combinatorial datum (we will write $T=(\pi,\lbd)$). 
We usually represent $\pi$ by a table of two lines (here $\mathcal A = \{1,\dots,d\}$):
$$
\pi=\left(\begin{array}{cccc}
1&2&\ldots&d \\ \pi^{-1}(1)&\pi^{-1}(2)&\ldots&\pi^{-1}(d)
\end{array}\right).
$$
When the  measured foliation is  not oriented, the  above construction
does not make sense. Nevertheless a  generalization of interval
exchange maps for any  measured foliation on  a surface  (oriented or
not) was introduced by Danthony and Nogueira~\cite{Danthony1988}. The 
generalizations  (\emph{linear involutions})  corresponding  to oriented
flat surfaces with $\Z/2\Z$ linear  holonomy were studied in detail by
Boissy and Lanneau~\cite{Lanneau:rauzy} (see also
Avila-Resende~\cite{Avila2008} for a similar construction). 
We  briefly  recall the objects here. \medskip
 
Roughly   speaking,  a  linear   involution  encodes   the  successive
intersections of  the foliation with some transversal  interval $I$.
We choose $I$ and a positive vertical direction (equivalently, a choice of left and right ends of $I$)
that intersect every vertical geodesics. The first return map $T_0:I\rightarrow I$ of vertical geodesics 
in the positive direction is well defined, outside a finite number of points (called
singular points) that correspond to vertical geodesics that stop at a singularity before 
intersecting again the interval $I$.  This equips $I$ with is a finite open partition $(I_\alpha)$
so that $T_0(x) = \pm x+t_\alpha$.

However the map $T_0$ alone does not properly correspond to the dynamics of  vertical geodesics
since when $T_0(x)=-x+t_\alpha$ on the interval $I_\alpha$, then $T_0^2(x)=x$,
and $(x,T_0(x), T_0^2(x))$ does not correspond to the successive
intersections of a vertical geodesic with $I$ starting from $x$. To fix
this problem, we have to consider $T_1$ the first return map of the
vertical geodesics starting from $I$  in the negative direction. Now if $T_0(x)=-x+c_i$ then the
successive intersections with $I$ of the vertical geodesic starting
from $x$ will be $x,T_0(x), T_1(T_0(x))$\ldots

\begin{Definition}
Let  $f$  be  the   involution  of  $I\times\{0,1\}$  given  by  $f(x,
\varepsilon)=(x,1-\varepsilon)$. A  \emph{linear involution} is  a map
$T$, from $I\times\{0,1\}$ into itself, of the form $f\circ \tilde T$,
where $\tilde  T$ is an  involution of $I\times\{0,1\}$  without fixed
point,  continuous except  on a  finite set  of point  $\Sigma_T$, and
which  preserves the  Lebesgue measure.  In  this paper  we will  only
consider linear  involutions with the  following additional condition:
the  derivative  of  $\tilde   T$  is  $1$  at  $(x,\varepsilon)$  if
$(x,\varepsilon)$ and $T(x,\varepsilon)$  belong to the same connected
component, and $-1$ otherwise.
\end{Definition}

\begin{Remark}
A linear  involution $T$ that preserves $I\times \{0\}$ corresponds precisely 
to an interval  exchange transformation map $T_0$ (by restricting  $T$ to
$I\times \{0\}$).
Therefore,  we can  identify the  set of interval exchange maps with a subset of the linear involutions.
\end{Remark}

As for interval exchange maps, a linear  involution $T$  is  encoded  by a  combinatorial  datum  called
\emph{generalized permutation} and by continuous data.  
This is done in  the following way: $I\times \{0\}\backslash \Sigma_T$
is a union of $l$  open intervals $I_1\sqcup \ldots \sqcup I_l$, where
we assume by  convention that $I_i$ is the interval  at the place $i$,
when  counted  from  the   left  to  the  right.  Similarly,  $I\times
\{1\}\backslash   \Sigma_T$  is   a  union   of  $m$   open  intervals
$I_{l+1}\sqcup \ldots \sqcup I_{l+m}$. For all $i$, the image of $I_i$
by  the map $\tilde{T}$  is a  interval $I_j$,  with $i\neq  j$, hence
$\tilde  T$ induces  an involution  without  fixed points  on the  set
$\{1,\ldots, l+m\}$.  To encode this involution, we  attribute to each
interval $I_i$ a symbol such  that $I_i$ and $\tilde T(I_i)$ share the
same symbol. Choosing the set of symbol to be $\mathcal A$, we get
a  two-to-one  map  $\pi:\{1,\ldots,l+m\}\rightarrow  \mathcal A$,
with $d:=|\Alp|=\frac{l+m}{2}$. Note that $\pi$ is not uniquely defined by $T$
since  we  can   compose  it  on  the  left   by  any  permutation  of the 
alphabet $\Alp$. The continuous data of $T$ is a real vector $\lbd=(\lbd_\alpha)_{\alpha\in \Alp}$ with positive entries, which records the lengths of the permuted intervals.

\begin{Definition}
A \emph{generalized permutation} of type  $(l,m)$, with $l+m=2d$, is a
two-to-one map  $\pi:\{1,\ldots,2d\}\rightarrow \mathcal A$ to some finite 
alphabet $\mathcal A$. We will usually  represent such generalized permutation by  a table of
two lines of symbols, with each symbol appearing exactly two times.
$$
\pi=\left(\begin{array}{ccc}
\pi(1) & \dots& \pi(l) \\
\pi(l+1) & \dots & \pi(l+m)
\end{array}\right). 
$$
We will call \emph{top} (respectively \emph{bottom}) the restriction of a generalized permutation 
$\pi$ to $\{1,\ldots,l\}$ (respectively $\{l+1,\ldots,l+m\}$) where $(l,m)$ is the type of $\pi$.
\end{Definition}

In  the table representation  of a  generalized permutation,  a symbol
might  appear  two  times in  the same line (see Example~\ref{ex:gp} below). Therefore, we do not necessarily have $l=m$. \medskip

\noindent {\bf Convention.} We will use the terminology generalized permutations for 
permutations that are {\it not} ``true'' permutations.

\subsection{Irreducibility}

We say that a linear involution $T=(\pi,\lbd)$ admits a {\em suspension } if there exists a quadratic differential $(X,q)$ such that $T$ is the linear involution defined by the first return map of the vertical foliation to a horizontal segment in $X$. If $\pi$ is a ``true'' permutation defined over $d$ letters, it is well known~\cite{Veech1982} that $T$ admits a suspension if and only if $\pi$  is irreducible, {\em i.e.} $\pi(\{1,\dots , k\}) \neq \{1, \dots , k\},\quad 1\le k \le d - 1$. It turns out that a similar characterization exists for  generalized permutations. For the convenience of the reader, we state this criterion here. 

\begin{Definition}
\label{def:irred}
We will say that $\pi$ is {\em reducible} if $\pi$ admits a decomposition
\begin{equation*}\quad
\left(\begin{array}{c|c|c}A\cup B & *** & D\cup B \\\hline A\cup C &
  *** & D \cup C  \end{array}\right), \ A,B,C,D \textrm{
  disjoint subsets of } \mathcal{A},
\end{equation*}
where the subsets $A,B,C,D$ are not all empty and 
one of the following statements holds
\begin{enumerate}
\item[i-] No corner is empty.
\item[ii-] Exactly one corner is empty and it is on the left.
\item[iii-] Exactly two corners are empty and they are both on the same side.
\end{enumerate}
A permutation that is not reducible is {\em irreducible}.
\end{Definition}

For example of irreducible and reducible permutations, see Claim~\ref{claim:1}.

\begin{Theorem}[\cite{Lanneau:rauzy} Theorem 3.2]
\label{thm:gen:perm:irrd}
Let $T=(\pi,\lbd)$ be a linear involution. Then $T$ admits a suspension if and only if the underlying generalized permutation $\pi$ is irreducible.
\end{Theorem}

\subsection{Rauzy induction}

The \emph{Rauzy induction} $\mathcal R(T)$ of a linear involution $T$ is
the first  return map of $T$ to a smaller interval $I'\times  \{0,1\}$, where
$I'\subsetneq I$. More precisely, if $T=(\pi,\lbd)$ and $(l,m)$ is 
the  type of  $\pi$, we  identify $I$  with the  interval  $[0,1)$. If
$\lbd_{\pi(l)}\neq  \lbd_{\pi(l+m)}$, then  the  Rauzy induction
$\mathcal R(T)$ of $T$ is  the linear involution obtained by the first
return map of $T$ to
$$
\bigl(0,\max\{1-\lbd_{\pi(l)}, 1-\lbd_{\pi(l+m)}\}\bigr) \times
\{0,1\}.
$$
It is easy to see that this is again a linear involution, defined on $d$ letters.

The combinatorial  data of the  new linear involution depends  only on
the      combinatorial      data      of     $T$      and      whether
$\lbd_{\pi(l)}>\lbd_{\pi(l+m)}$       or      $\lbd_{\pi(l)}<
\lbd_{\pi(l+m)}$.  We  say  that  $T$  has type  $0$  or  type  $1$
respectively. The combinatorial data of $\mathcal{R}(T)$ only
depends on $\pi$ and on the type of $T$. This defines two operations
$\mathcal{R}_0$ and $\mathcal{R}_1$ by
$\mathcal{R}(T)=(\mathcal{R}_\varepsilon(\pi), \lbd^{\prime})$,
with  $\varepsilon$  the  type  of $T$  (see~\cite{Lanneau:rauzy}  for
details). We will not use these operations in this paper. \medskip

We stress that the Rauzy-Veech induction is well defined if and only if the two
rightmost intervals do not have  the same length {\em i.e.} $\lbd_{\pi(l)}\neq  \lbd_{\pi(l+m)}$.
However, when  these intervals do have  the same length,  we can still
consider the first return map of $T$ to
$$
\bigl(0,1-\lbd_{\pi(l)}\bigr) \times \{0,1\}.
$$
This is again a linear involution, denoted by $\mathcal R_{sing}(T)$, 
defined over $d-1$ letters. The combinatorics of $\mathcal R_{sing}(T)$ can 
be defined as follows: we apply the top operation of the Rauzy induction 
and then we erase the last letter of the top. Equivalently, we apply the bottom 
operation of the Rauzy induction and then we erase the last letter of the bottom.

\begin{Example}
\label{ex:gp}
Let $T=(\pi,\lbd)$ with 
$\pi = \left(\begin{smallmatrix} A & A & B &C \\ D& C &B&D  \\ \end{smallmatrix}\right)$.
Then the combinatorial datum of the Rauzy induction $\mathcal R(T)$ of $T$ is:
$$
\begin{array}{lll}
\left(\begin{smallmatrix}  A & A & B &C \\ D& C &D&B   \\ \end{smallmatrix}\right) & \textrm{if} & \lbd_{C} > \lbd_{D} \\ \\
\left(\begin{smallmatrix}  A & A & B \\ C& D& C &B&D   \\ \end{smallmatrix}\right) & \textrm{if} & \lbd_{C} < \lbd_{D} \\ \\
\left(\begin{smallmatrix}  A & A & B \\ D& D &B   \\ \end{smallmatrix}\right) & \textrm{if} & \lbd_{C} = \lbd_{D}
\end{array}
$$
\end{Example}

We can formally define the converse of the (singular) Rauzy Veech operations. We proceed as 
follows: given some permutation $\pi'$ defined over an alphabet $\Alp$, and a letter $\alpha\not\in \Alp$,
we put $\alpha$ at the end of the 
top or bottom line of $\pi'$. Then we choose some letter $\beta\in \Alp$. We 
replace $\beta$ by $\alpha$ and we put the letter $\beta$ at the end of the bottom line of $\pi'$.
The new permutation $\pi$ we have constructed is defined over the alphabet  $\Alp \sqcup \{\alpha\}$
and satisfies $\mathcal R_{sing}(\pi) = \pi'$. It turns out that all the permutations of 
$\mathcal R^{-1}_{sing}(\pi')$ are constructed as above with one exception: the one given by putting at the 
end of the top and bottom line the letter $\alpha$.

\begin{Example}
\label{ex:rauzy:inverse}
Let $\pi'=\left(\begin{smallmatrix} A& B& C& D \\B& A &D& C   \\ \end{smallmatrix}\right)$.
For instance if we choose the letter $\beta=B$ and we put $\alpha$ at the end of the top line, 
we get $\left(\begin{smallmatrix} A & \alpha & C& D & \alpha \\ B& A & D & C & B       \\ \end{smallmatrix}\right)$
or $\left(\begin{smallmatrix} A& B& C& D & \alpha \\ \alpha & A &D & C & B   \\ \end{smallmatrix}\right)$.
More precisely, if $\beta$ ranges over all letters of $\Alp$ we get 
(up to a permutation of the letters of the alphabet $\Alp \sqcup \{\alpha\}$):
\begin{multline*}
\mathcal R^{-1}_{sing} \left(\begin{smallmatrix} A& B& C& D \\B& A &D& C   
\\ \end{smallmatrix}\right) = \{  
\left(\begin{smallmatrix} A& B& C& D & \alpha \\ \alpha & A &D & C & B   \\ \end{smallmatrix}\right),
\left(\begin{smallmatrix} A& B& C& D & \alpha \\ B& \alpha &D& C & A      \\ \end{smallmatrix}\right),
\left(\begin{smallmatrix} A& B& C& D & \alpha \\ B& A & \alpha &D& C       \\ \end{smallmatrix}\right),
\left(\begin{smallmatrix} A& B& C& D & \alpha \\ B& A & D & \alpha & C       \\ \end{smallmatrix}\right),
\left(\begin{smallmatrix} A& B& C& D & \alpha \\ B& A & D & C & \alpha       \\ \end{smallmatrix}\right),
\left(\begin{smallmatrix} A& B& C& \alpha& D \\ B& A & D & C & \alpha       \\ \end{smallmatrix}\right), \\
\left(\begin{smallmatrix} \alpha & B& C& D & \alpha \\ B& A & D & C & A       \\ \end{smallmatrix}\right),
\left(\begin{smallmatrix} A & \alpha & C& D & \alpha \\ B& A & D & C & B       \\ \end{smallmatrix}\right),
\left(\begin{smallmatrix} A & B & \alpha & D & \alpha \\ B& A & D & C & C       \\ \end{smallmatrix}\right),
\left(\begin{smallmatrix} A & B & C & \alpha & \alpha \\ B& A & D & C & D       \\ \end{smallmatrix}\right)
\}.
\end{multline*}

\end{Example}

We end this section with the following easy lemma that will be useful for the sequel.

\begin{Lemma}
\label{lm:rauzy:sing}
If $\pi$ is an irreducible generalized permutation then $\mathcal R_{sing}(\pi)$
is also a generalized permutation.
\end{Lemma}

\subsection{Rauzy induction and SAF-invariant}

We can naturally  extend $SAF(.)$ to linear involutions 
by $SAF(T) := SAF(\hat T)$  where the transformation $\hat T$ is the
double of $T$ (see~\cite{Avila2008} for details). As for interval exchange maps,
if $T$ is periodic then $SAF(T)=0$. The  converse is true  if $|\mathcal  A|=d\leq
3$ (see Lemma~\ref{lm:d=3} for a proof in case $d=3$).

\begin{Proposition}
\label{prop:saf:n=2}
A linear involution $T$ defined over $d \leq 2$ intervals is completely periodic if and 
only if $SAF(T)=0$.
\end{Proposition}

\begin{proof}
We  fix  some  alphabet  $\mathcal  A  =  \{A,B\}$.  There  are  three
possibilities  depending  on   the  combinatorics  of  the  associated
permutation: 
$$
\left(\begin{smallmatrix} A &A \\ B &B \\ \end{smallmatrix}\right),\
\left(\begin{smallmatrix} A &B \\ A &B \\ \end{smallmatrix}\right),\
\left(\begin{smallmatrix} A &B \\ B &A \\ \end{smallmatrix}\right).
$$
In the first two cases, clearly $T$ is completely periodic (no matter what $SAF(T)$ is).
In the last case $T$ is a rotation of $[0,1)$ of some angle $\theta$:
$$
T(x) = \left\{ \begin{array}{ll}
x + \theta & \textrm{if }  0<x<1-\theta, \\
x + \theta-1 & \textrm{if }  1-\theta<x<1.
\end{array} \right.
$$
Direct computation shows $SAF(T) = 2\wedge_\Q\theta$; hence $SAF(T)=0$ implies $\theta\in \Q$ and 
$T$ is completely periodic.
\end{proof}
Since the  SAF-invariant is a scissors congruence invariant, it is preserved by the Rauzy operations.
\begin{Proposition}
\label{prop:rauzy:cp}
Let $T$ be a linear involution. If the Rauzy induction is well defined then $T$ then 
$$
SAF(\mathcal R(T)) = SAF(T).
$$
\noindent Otherwise (if the Rauzy induction is not well defined), then
$$
SAF(\mathcal R_{sing}(T)) = SAF(T).
$$

\noindent Moreover, $T$ is completely periodic if and only if $\mathcal R(T)$ or ${\mathcal R}_{sing}(T)$ is completely periodic. 
\end{Proposition}

\subsection{Rauzy induction and Keane property}
\label{subsec:keane}

We will say that $T=(\pi,\lbd)$ is {\it decomposed} if there exists $\ \mathcal A' \subsetneq A$
such that both following conditions hold:
$$
\left\{ \begin{array}{l}
\pi = \left(\begin{smallmatrix} \alpha_1&\ldots&\alpha_{i_0}&|& * * * \\
\beta_1&\ldots &\beta_{j_0} & | & * * *  \\ \end{smallmatrix}\right), \textrm{ where }
\{\alpha_1,\ldots,\beta_{j_0}\}=\mathcal A'' \sqcup \mathcal A'', \\
\sum_{i=1}^{i_0}\lbd_{\pi(i)}=\sum_{j=1}^{j_0}\lbd_{\pi(j)}, 
\textrm{for some } 1\leq i_0 < l$ and $1\leq j_0 < m.
\end{array}
\right.
$$
This means exactly that $T$ splits into two linear involutions: $T=T_{1}\comp T_{2}$.
Since the SAF-invariant is additive we have 
\begin{equation}
\label{eq:saf}
SAF(T) = SAF(T_{1})+SAF(T_{2}).
\end{equation}

\begin{Definition}
A linear involution has a connection (of length $r$) if there exist $(x,\varepsilon)\in
I\times \{0,1\}$ and $r\geq 0$ such that
\begin{itemize}
\item $(x,\varepsilon)$ is a singularity for $T^{-1}$.
\item $T^r(x,\varepsilon)$ is a singularity for $T$.
\end{itemize}
A linear involution with no connection is said to have the {\em Keane property}
(also called the infinite distinct orbit condition or i.d.o.c. property).
\end{Definition}

An instance of a linear involution with a connection of length $1$ is when $T$ is decomposed.

If $T=(\pi,\lbd)$ is a linear involution, we will use the notation
$(\pi^{(n)},\lbd^{(n)}):=\mathcal R^{(n)}(T)$ if the $n$-th iteration
of $T$ by $\mathcal R$ is well defined, and $\lbd_\alpha^{(n)}$
for the length of the interval associated to the symbol $\alpha\in
\mathcal A$. The next proposition is a slightly more precise statement of 
Proposition~$4.2.$ of~\cite{Lanneau:rauzy}:

\begin{Proposition}
\label{prop:BL}
The following statements are equivalent.
\begin{enumerate}
\item $T$ satisfies the Keane property.
\item $\mathcal R^{(n)}(T)$ is well defined for any $n\geq 0$ and the lengths 
of the intervals $\lbd^{(n)}$ tends to $0$ as $n$ tends to infinity.
\item The transformation $T$ is minimal
\end{enumerate}
\noindent More precisely, if $T$ has a connection and if the Rauzy induction $\mathcal R^{(n)}(T)$ 
is well defined for every $n\geq 0$, then there exists $n_{0}>0$ such that 
$\mathcal R^{(n_{0})}(T)$ is decomposed.
\end{Proposition}

The following proposition relates connections with vanishing SAF-invariant.

\begin{Proposition}
\label{prop:connection}
Let $T$ be  a linear involution such that  the lengths of  the exchanged intervals belong to 
a $2$-dimensional space over $\Q$. If $SAF(T)=0$ then $T$ has a connection.
\end{Proposition}

\begin{proof}
Let $\hat T$ be the double of $T$. Since the interval exchange map $\hat T$ 
has vanishing SAF, $\hat T$ is not ergodic (see Arnoux's thesis~\cite{Ar1}). But
by a result of Boshernitzan (\cite{Boshernitzan1988}, Theorem~1.1), $\hat T$ is neither 
minimal (otherwise $\hat T$ would be uniquely ergodic).

\noindent If $T$ has no connection then it satisfies Keane property
and Proposition~\ref{prop:BL} implies $T$ would be minimal. So that 
$\hat T$ would also be minimal that is a contradiction.
\end{proof}

\section{Complete periodicity of linear involution up to $5$ intervals}
\label{subsec:cp} 

In this section we specialize the analysis of complete periodicity to 
linear involutions. In  the sequel,  $T$ will  be a  linear involution  defined  over $d$
intervals. We prove several lemmas depending on
the values of $d\in \{3,\dots,6\}$.  Section~\ref{subsec:cp} is devoted to
the case $d\leq 5$; as a corollary we will draw Theorem~\ref{thm:CAP:to:CP}.
In Section~\ref{sec:d=6} we will consider the case $d=6$ and 
deduce Theorem~\ref{theo:cp:stronger}. \medskip

Since we proceed by induction on $d$, let us start with the case $d=3$.

\begin{Lemma}[$d=3$]
\label{lm:d=3}
If $T$ is a linear involution defined over $3$ intervals with $SAF(T)=0$ then
$T$ is completely periodic.
\end{Lemma}

\begin{proof}
The condition $SAF(T)=0$ implies that the lengths of the intervals exchanged by $T$ span a $2$-dimensional space over $\Q$. It follows from 
Proposition~\ref{prop:connection} $T$ has a connection.  Hence from
Proposition~\ref{prop:BL}, two possibilities can occur:
\begin{itemize}
\item[(a)] either the Rauzy induction $\mathcal R^{(n_{0})}(T)$ is not well defined for some $n_{0}> 0$, or
\item[(b)] there exists $n_{0}>0$ such that the permutation $\mathcal R^{(n_{0})}(T)$ is 
decomposed.
\end{itemize}

Let us first consider case (a). Since the Rauzy induction is not well defined 
on $T' := \mathcal R^{(n_{0})}(T)$, there is a relation: $\lbd_{\pi^{(n_{0})}(l)} =  \lbd_{\pi^{(n_{0})}(l+m)}$.
We can consider the first return map of $T'$ to 
$$
\bigl(0,1-\lbd_{\pi^{(n_{0})}(l)}\bigr) \times \{0,1\}.
$$
We get a new $T'' = \mathcal R_{sing}(T')$ defined over $2$ intervals. Since
$0=SAF(T)=SAF(T')=SAF(T'')$ we get that $T''$ has vanishing SAF and hence
$T''$  is completely periodic.
We conclude by Proposition~\ref{prop:rauzy:cp} that $T'$ and $T$ are also completely 
periodic. \medskip 

Let  us now consider case  (b). By assumption, $T$  splits as $T=T_{1}
\comp T_{2}$. We will denote the alphabet by $\mathcal A = \{A,B,C\}$.
The decomposition of the permutation $\pi^{(n_{0})}$ involves the following 
decomposition (up to permutation of the letters and $T_{i}$):
$$
\left(\begin{array}{c|c}A& \alpha_{1} \ \alpha_{2} \\ A & \beta_{1} \ \beta_{2} 
\end{array}\right), \textrm{ where } \{\alpha_{1},\dots,\beta_{2} \} = \{B,C\} \sqcup\{B,C\}.
$$
Here $T_{1}= \left( \left(\begin{array}{c}A \\ A \end{array}\right), \lbd_{A}\right)$.
Thus obviously $SAF(T_{1})=0$. Reporting into Equation~(\ref{eq:saf}):
$$
0=SAF(T) = SAF(T_{1})+SAF(T_{2}),
$$
we draw $SAF(T_{2})=0$. Since $T_{2}$ is a linear involution defined over $2$ letters, 
we again conclude that $T_2$ is completely periodic. This proves the lemma.
\end{proof}

In the next lemma, we continue this induction process. The idea is to consider the inverse of the singular Rauzy induction. Thus the number of 
intervals increases, and we need to avoid `` bad '' cases.

\begin{Lemma}[$d=4$]
\label{lm:d=4}
If $T=(\pi,\lbd)$ is a linear involution defined over $4$ intervals with $SAF(T)=0$,
and if $\pi \not = \left(\begin{smallmatrix} A& B& C& D \\B& A &D& C   \\ \end{smallmatrix}\right)$
up to a permutation of the letters, then $T$ is completely periodic.
\end{Lemma}

\begin{proof}
We first show that $T$ has a connection. We can assume that $\pi$ is irreducible and the 
Rauzy induction $\mathcal R^{(n)}(T)$ is well defined for any $n>0$
(otherwise $T$ would have a connection and we are done). 

If $\pi$ is a true permutation then, up to  Rauzy inductions, 
$\pi=\left(\begin{smallmatrix} A& B&C&D \\D&C&B & A \end{smallmatrix}\right)$.
For any $\alpha \in \Alp$, one can compute the translation lengths $t_\alpha$  in terms 
of the lengths of the subintervals, namely 
$$
(t_A,t_B,t_C,t_D) = (\lbd_B + \lbd_C + \lbd_D,\lbd_C + \lbd_D - \lbd_A,
\lbd_D - \lbd_A - \lbd_B,-\lbd_A - \lbd_B - \lbd_C).
$$
It follows
$$
0 = SAF(T) = \sum_{\alpha\in \Alp}\lbd_\alpha\wedge_\Q t_\alpha = \lbd_A \wedge (\lbd_B+\lbd_C+\lbd_D)
+ \lbd_B \wedge (\lbd_C+\lbd_D) + \lbd_C\wedge \lbd_D.
$$
From the relation $-\lbd_C\wedge \lbd_D = \lbd_A \wedge (\lbd_B+\lbd_C+\lbd_D) + \lbd_B \wedge (\lbd_C+\lbd_D)$
we draw the new one
\begin{equation}
\label{eq:1:wedge}
\lbd_A \wedge \lbd_B \wedge \lbd_C \wedge\lbd_D=0.
\end{equation}
Similarly, from $-\lbd_A \wedge \lbd_B = \lbd_A \wedge(\lbd_C+\lbd_D) + \lbd_B \wedge (\lbd_C+\lbd_D) + 
\lbd_C\wedge \lbd_D$, we get:
\begin{equation}
\label{eq:2:wedge}
\lbd_A\wedge \lbd_B \wedge (\lbd_C+\lbd_D)= 0.
\end{equation}
Hence Equations~\eqref{eq:1:wedge} and~\eqref{eq:2:wedge} show that the lengths of  the exchanged 
intervals have linear rank $2$ over $\Q$. By Proposition~\ref{prop:connection}, the linear involution
$T$ thus has a connection. \medskip

Now, if $\pi$ is a generalized permutation, since it is irreducible, there is a 
suspension $(Y,q)$ belonging to some stratum of quadratic differentials and inducing 
$\pi$. Since the number of intervals is $4$, the dimension of the stratum is $3$. Hence
the only possibility is $\QQQ(-1,-1,2)$ and, up to the Rauzy induction, the permutation is
$\pi=\left(\begin{smallmatrix} A& A&B&C \\C&B&D&D \end{smallmatrix}\right)$.
The same computation shows that the lengths of the exchanged intervals have linear rank 
$2$ over $\Q$ so that Proposition~\ref{prop:connection} applies and $T$ has a connection. \medskip

We now repeat the same strategy as the proof of the previous lemma. From Proposition~\ref{prop:BL} only two possibilities 
(a) and~(b) can occur (see the proof of Lemma~\ref{lm:d=3}). In case~(a), the Rauzy induction is not well defined and we can reduce the 
problem to some $T'=\mathcal R^{(n_{0})}(T)$ defined over $3$ letters. We then 
conclude using Lemma~\ref{lm:d=3}. 

\noindent Thus let us assume that there exists $n_{0}>0$ such that $\mathcal R^{(n_{0})}(T)$
breaks into two linear involutions $T_{1}$ and $T_{2}$ with $SAF(T_{1}) + SAF(T_{2})=0$. 
Again if $T_{1}$ or $T_{2}$ is defined over only $1$ interval then we are done 
(by the same argument as above). So assume that $T_{1}$ and $T_{2}$ are defined over $2$ intervals. 
If $SAF(T_{1})=0$ then we are done by Proposition~\ref{prop:saf:n=2}. Hence we will assume
that $SAF(T_{1}) = -SAF(T_{2}) \neq 0$ and we will get a contradiction. 
This can be achieved only if the permutation associated to $T_{1}$ has the form 
$\left(\begin{smallmatrix} A& B \\B & A \end{smallmatrix}\right)$.
The same is true for $T_{2}$: the permutation is $\left(\begin{smallmatrix} C& D \\D & C \end{smallmatrix}\right)$.
Hence $\pi^{(n_{0})} = \left(\begin{smallmatrix} A& B& C& D \\B& A &D& C   \\ \end{smallmatrix}\right)$.
From this observation, it is not hard to see that $\pi = \pi^{(n_{0})}$, that is the desired contradiction.
The lemma is proven.
\end{proof}

For the case $d=5$, again new pathological cases appear as shown in the next lemma:

\begin{Lemma}[$d=5$]
\label{lm:d=5}
Let $T=(\pi,\lbd)$ be a linear involution defined over $5$ intervals.
We will assume that $\pi$ does not belong to one of the following sets 
(up to permutation of the letters of $\mathcal A$):
$$
\begin{array}{l}
\mathcal E_{1}=\left\{  
\left(\begin{smallmatrix} A& B& C& D & \alpha \\ \alpha & A &D & C & B   \\ \end{smallmatrix}\right),
\left(\begin{smallmatrix} A& B& C& D & \alpha \\ B& \alpha &D& C & A      \\ \end{smallmatrix}\right)
\right\}, \qquad
\mathcal E_{2}=\left\{
\left(\begin{smallmatrix} \alpha & B& C& D & \alpha \\ B& A & D & C & A       \\ \end{smallmatrix}\right),
\left(\begin{smallmatrix} A & \alpha & C& D & \alpha \\ B& A & D & C & B       \\ \end{smallmatrix}\right)
\right\}, \\
\mathcal E_{3}=\left\{
\left(\begin{smallmatrix}
\begin{smallmatrix}A\ B \\  B\ A  \end{smallmatrix} &\ \pi'
\end{smallmatrix}\right),
\left(\begin{smallmatrix}
\ \pi' & \begin{smallmatrix}A\ B \\  B\ A  \end{smallmatrix}
\end{smallmatrix}\right), \textrm{ $\pi'$ permutation defined over $3$ letters}
\right\}.
\end{array}
$$
If $SAF(T)=0$, and $T$ has a connection, then $T$ is completely periodic.
\end{Lemma}

\begin{proof}[Proof of Lemma~\ref{lm:d=5}]
Imitating the proofs of the two preceding lemmas, up to replacing $T$ by $\mathcal R^{(n)}(T)$
for some suitable $n$, we can assume that either
\begin{itemize}
\item[(a)] the Rauzy induction is not well defined for $T$, or
\item[(b)] $T$ is decomposed. 
\end{itemize}
In case~(a), we get new linear involution $T'=(\pi',\lbd')=\mathcal R_{sing}(T)$
defined over $4$ intervals with vanishing SAF-invariant. If 
$\pi' \not = \left(\begin{smallmatrix} A& B& C& D \\B& A &D& C   \\ \end{smallmatrix}\right)$
then we are done by preceding Lemma~\ref{lm:d=4}. Hence let us assume that
$\mathcal R_{sing}(\pi) = \left(\begin{smallmatrix} A& B& C& D \\B& A &D& C   
\\ \end{smallmatrix}\right)$, and we will get a contradiction. But by Example~\ref{ex:rauzy:inverse}
permutations in $\mathcal R^{-1}_{sing} \left(\begin{smallmatrix} A& B& C& D \\B& A &D& C   
\\ \end{smallmatrix}\right)$ are exactly those lying in 
$\mathcal E_{1} \cup \mathcal E_{2} \cup \mathcal E_{3}$. \medskip

\noindent  In case~(b),  $T=T_{1}\comp T_{2}$  is decomposed  into two
linear  involutions, where $T_i$ is defined  over $d_i$  intervals with $d_1+d_2=5$,
and satisfy $SAF(T_{1}) + SAF(T_{2})=0$. There are two possible 
partitions for $\{d_1,d_2\}$, namely $\{1,4\}$ or $\{2,3\}$. In the first situation
$T_1$ (resp. $T_2$) is completely periodic as any linear involution 
defined over only $1$ interval is periodic. Hence $SAF(T_2)=0$ (resp. 
$SAF(T_1)=0$). One wants to use Lemma~\ref{lm:d=4}. For that we need 
to avoid the case $\pi_i = \left(\begin{smallmatrix} A& B& C& D \\B& A &D& C   \\ \end{smallmatrix}\right)$.
But in the latter situation one would have
$$
\pi = \left(\begin{smallmatrix}
E & A & B & C & D \\
E & B & A & D & C
\end{smallmatrix}\right) \qquad \textrm{respectively,} \qquad
\pi = \left(\begin{smallmatrix}
A & B & C & D & E \\
B & A & D & C & E
\end{smallmatrix}\right).
$$
Those permutations belong to the set $\mathcal E_{3}$ and we are done.

\noindent For the case $\{d_1,d_2\}=\{2,3\}$. Assume that $d_1=2$. If $SAF(T_{1})=0$ then we are done. 
On the other hand $SAF(T_{1})\not =0$ implies $\pi_1=\left(\begin{smallmatrix} A
& B \\ B & A  \end{smallmatrix}\right)$, that is $\pi \in \mathcal E_{3}$. The same is true
if $T_{2}$ is the linear involution defined over $2$ letters. The lemma is proved.
\end{proof}

\section{Complete Algebraic Periodicity implies Complete Periodicity}
\label{sec:prf:CAP:CP}
We begin with the following simple lemma.

\begin{Lemma}
\label{lm:periodic:orbit}
Let $T=(\pi,\lbd)$ be a linear involution defined over $3$ letters. We assume 
that $\pi$ is a generalized permutation. If $T$ has a periodic orbit then $T$ is 
completely periodic.
\end{Lemma}

We postpone the proof of the lemma to the end of this section and show Theorem~\ref{thm:CAP:to:CP}.

\begin{proof}[Proof of Theorem~\ref{thm:CAP:to:CP}]
Let $\theta$ be a direction of a cylinder in $X$. The core curve of this cylinder represents an element of $H_1(X,\Z)$,  hence by assumption, the SAF-invariant of the foliation $\mathcal{F}_\theta$ vanishes. 
As usual one assumes that $\theta$ is the vertical direction. We want to show that the flow in the vertical 
direction is periodic. Let $T=(\pi,\lbd)$ be the linear involution given by the cross section of the vertical foliation
to some full transversal interval. By assumption $T$ is defined over $6$ intervals and has a periodic orbit.
Moreover $\pi$ is an {\em irreducible} generalized permutation. \medskip

Obviously proving complete periodicity for the vertical foliation or for $T$ is the same.
Since $T$ has a periodic orbit Proposition~\ref{prop:BL} implies that only two cases can occur 
(up to replacing $T$ by $\mathcal R^{(n)}(T)$ for some suitable $n$), either
\begin{itemize}
\item[(a)] the Rauzy induction $\mathcal R(T)$ is not well defined, or
\item[(b)] $\mathcal R(T)$ is decomposed.
\end{itemize}
\noindent {\bf Case (a).} Since the Rauzy induction is not well defined $T'=\mathcal R_{sing}(T)=(\pi',\lbd')$
is a linear involution defined over $5$ letters, with a periodic orbit and vanishing SAF-invariant. 
In view of applying Lemma~\ref{lm:d=5},
if $\pi'\not \in \mathcal E_{i}$ for all $i$ then $T'$ is completely periodic and we are done. 
Hence let us assume $\pi' \in \mathcal E_{1}\sqcup\mathcal E_{2}\sqcup\mathcal E_{3}$.
Recall that $\pi'$ is a generalized permutation (by Lemma~\ref{lm:rauzy:sing}) thus $\pi'\not \in \mathcal E_1$.
If $\pi'\in \mathcal E_2$ then 
$\pi' = \left(\begin{smallmatrix} \alpha & B& C& D & \alpha \\ B& A & D & C & A  \end{smallmatrix}\right)$. 
Observe that $\lbd'_\alpha = \lbd'_A$. So that applying once more the singular Rauzy induction we get:
$$
\mathcal R_{sing}(T') =T''= (\pi'',\lbd''), \qquad \textrm{where} \qquad \pi'' = 
\left(\begin{smallmatrix}  A& B& C& D \\ B& A & D & C  \end{smallmatrix}\right).
$$
In particular $T''$ is decomposed. Since it has a periodic orbit, we see that it is completely 
periodic.

Now if $\pi'\in \mathcal E_3$ then $T'$ splits into two linear involutions $T'_1$ and $T'_2$ where 
(up to permuting the indices) $\pi'_1 = \left( \begin{smallmatrix}A& B \\  B& A  \end{smallmatrix}\right)$
and $\pi'_2$ is a {\it generalized} permutation defined over $3$ letters. Since $T$ has a periodic 
orbit, it follows that one of the $T_i$ has a periodic orbit.  We then conclude by using Lemma~\ref{lm:periodic:orbit}. \medskip

\noindent {\bf Case (b).} In this situation, $T=(\pi,\lbd)=T_{1}\comp T_{2}$  is decomposed  into two
linear  involutions, each defined  over $d_i$  intervals with $d_1+d_2=6$,
with opposite $SAF$. There are three possible (unordered) partitions for $\{d_1,d_2\}$, namely 
$\{1,5\}$, $\{2,4\}$ or $\{3,3\}$. In the first situation $\pi$ is reducible that is a contradiction.
In the second situation since $\pi$ is irreducible, we necessarily have 
$\pi_1 = \left(\begin{smallmatrix} A\ A \\  B\ B  \end{smallmatrix} \right)$. Hence $T_1$ is completely 
periodic, and $SAF(T_2)=0$. We conclude with Lemma~\ref{lm:d=4} ($\pi_2$ is not 
$\left(\begin{smallmatrix} C& D& E& F \\D& C &F& E   \\ \end{smallmatrix}\right)$ otherwise $\pi$
would be reducible). In the last case $T_1$ or $T_2$ has a closed orbit, say $T_1$. Again, 
by the irreducibility of $\pi$, the two permutations $\pi_1$ and $\pi_2$ are generalized permutations.
Then Lemma~\ref{lm:periodic:orbit} implies that $T_1$ is completely periodic. Hence $SAF(T_2)=0$
and we conclude by Lemma~\ref{lm:d=3}.

Thus the vertical flow on $(X,\omega)$ is periodic, and Theorem~\ref{thm:CAP:to:CP} is proven.
\end{proof}

\begin{proof}[Proof of Lemma~\ref{lm:periodic:orbit}]
We very briefly describe the argument (based on the same strategy as above). Since $T$ is not 
minimal, up to replacing $T$ by some of its iterates under the Rauzy induction, 
either $\mathcal R(T)$ is not well defined, or $T$ is decomposed. In the 
first case the problem reduces to some $T'$ (defined over $2$ intervals) with a periodic orbit,
hence we are done. In the latter case $T$ decomposes as two linear involutions $T_i$. 
Since $\pi$ is a generalized permutation by assumption, one of the permutations $\pi_i$ is of the form 
$\left(\begin{smallmatrix} A\ A \\  B\ B \end{smallmatrix} \right)$. Hence the corresponding $T_i$
is completely periodic and we are done.
\end{proof}

\section{Complete algebraic periodicity implies real multiplication}
\label{sec:prf:CAP:to:eigen}

The aim of this section is to prove the converse of Theorem~\ref{thm:Alg:Periodic}, namely 
Theorem~\ref{thm:CAP:eigen}. Our proof is based on the following theorems.

\begin{Theorem}[McMullen, \cite{Mc6} Theorem~3.5]
\label{theo:hyp:Prym:eig}
Let $(X,\omega)$  be a Prym form with $\dim_\C{\rm Prym}(X)=2$, and assume here is a hyperbolic element $A$ in $\SL(X,\omega)$, where $\SL(X,\omega)$ denotes the Veech group of $(X,\omega)$. Then  $(X,\omega)$ is a Prym eigenform in $\Omega E_D$, for some discriminant $D$ satisfying $\Q(\sqrt{D})=\Q({\rm Tr}(A))$.
\end{Theorem}

\begin{proof}[Sketch of proof]
Let  $\phi   :  X  \rightarrow  X$   be  a pseudo-Anosov affine with respect to  the flat metric given by $\omega
\in \Omega(X,\inv)^-$ ({\em e.g.} given by Thurston's construction~\cite{Thurston1988}). 
By replacing $\phi$ by  one of its powers if necessary, we can assume that $\phi$ commutes with $\inv$. It follows that $\phi$ induces an isomorphism of $H_1(X,\Z)^-$ preserving the intersection form. Therefore

$$ T = \phi_{\ast}  + \phi_{\ast}^{-1} : H_1(X,\Z)^{-} \longrightarrow
H_1(X,\Z)^{-},$$

\noindent is a self-adjoint endomorphism of $\Prym(X,\inv)$.  Observe  that $T$ preserves the complex line
$S$   in    $(\Omega(X,\inv)^{-})^*$   spanned   by    the   dual   of
$\rm{Re}(\omega)$ and $\rm{Im}(\omega)$, and the restriction of $T$ to
this vector space is $\rm{Tr}(D\phi)\cdot \textrm{id}_{S}$. Since $\dim_\C
\Omega(X,\inv)^{-} = 2$, one has $\dim_\C S^{\perp} = 1$. But $T$ preserves  
the  splitting $(\Omega(X,\inv)^{-})^*  = S\oplus S^{\perp}$,
 and acts by real scalar multiplication on each line, hence  $T$ is  $\C$-linear,  {\em i.e.}  $T\in \mathrm{End}(\Prym(X,\inv))$.  This
equips $\Prym(X,\inv)$  with the real multiplication  by $\Z[T] \simeq
\mathcal O_{D}$ for a convenient discriminant $D$. Since $T^\ast\omega
= \rm{Tr}(D\phi)\omega$, the form $\omega$ becomes an eigenform for this
real multiplication. Observe that $\Q(\sqrt{D}) = \Q(\lbd+\lbd^{-1})$ where 
$\lbd$ being the expanding factor of the map $\phi$. Note that the fact $T\not\in \Z\mathrm{Id}$ follows from basic results in the theory of pseudo-Anosov homeomorphisms.
\end{proof}


\begin{Theorem}[Calta~\cite{Calta04}]
\label{theo:Cal:eq}
Fix a real quadratic field $K\subset \R$. Let $(X,\omega)$ be a completely algebraically periodic translation surface such that all the periods (both relative and absolute) of $\omega$ belong to $K(\imath)$.  Suppose that $(X,\omega)$ cannot be normalized  by $\GL^+(2,K)$ such that all the absolute periods of $\omega$ belong to $\Q(\imath)$.
Then if $(X,\omega)$ admits a decomposition into $k$ cylinders in the horizontal direction, then the following equality holds
\begin{equation}\label{eq:Calta1}
\sum_{i=1}^k w_ih'_i=0
\end{equation}

\noindent where $w_i,h_i$ are respectively the width and the height of the $i$-th cylinder, and $h'_i$ is the Galois conjugate of $h_i$ in $K$.
\end{Theorem}

\begin{Remark}
This statement is slightly more general than the statements~\cite[Proposition~4.1, and Lemma~4.2]{Calta04}
but its proof is essentially the same. One can also remark that Equation~\eqref{eq:Calta1} is the same as the one in 
Corollary~\ref{cor:mod:relation}.
\end{Remark}

\begin{proof}[Sketch of proof]
We have $K=\Q(\sqrt{f})$, where $f$ is a square-free positive integer. Recall that if $a,b \in \Q$ then 
$(a+b\sqrt{f})'=a-b\sqrt{f}$. Let $w,h\in K$.
\begin{eqnarray*}
4w\wedge h &=& (w+w'+w-w')\wedge (h+h'+h-h') \\
            &=& \frac{1}{\sqrt{f}}((w+w')(h-h')-(w-w')(h+h'))1\wedge\sqrt{f} \\
            &=& \frac{2}{\sqrt{f}}(w'h-wh')1\wedge\sqrt{f}.\\
\end{eqnarray*}
Let $C_i, \; i=1,\dots,k$, denote the cylinders in the horizontal direction. We identify each $C_i$ with a 
parallelogram $P_i$ in $\R^2$ which is constructed from the pair of vectors $\{(w_i,0),(t_i,h_i)\}$ in $K^2$. 
We have (see Section~\ref{subsec:J-inv}):
$$
J(X,\omega)=2\sum_{i=1}^k\left(\begin{array}{c} w_i \\ 0 \end{array}\right) \wedge \left(\begin{array}{c} t_i \\ h_i \end{array}\right).
$$
By assumption, the vertical direction $(0:1)$ is algebraically periodic. Hence
$$
J_{xx}(X,\omega)= \sum_{i=1}^k w_i\wedge t_i=0.
$$
Let $\overrightarrow{v}_q=(1,q)$ with $q\in K$, and $A_q= \left(\begin{array}{cc}1 & -1/q \\ 0 & 1/q \end{array}\right)$
so that $A_q\cdot\overrightarrow{v}_q=(0,1)$. Thus 
$$
J_{xx}(A_q\cdot(X,\omega))= \sum_{i=1}^k w_i\wedge(t_i-\frac{1}{q}h_i)=0.
$$
It follows that
$$\sum_{i=1}^kw_i\wedge sh_i=0, \; \forall s\in K,
$$
which implies
$$
\sum_{i=1}^k w'_ish_i-w_is'h'_i=0, \; \forall s\in K.
$$
By evaluating the last equality for $s=1$ and $s=\sqrt{f}$, we get
$$
\sum_{i=1}^k w_ih'_i=\sum_{i=1}^k w'_ih_i=-\sum_{i=1}^k w'_ih_i.
$$

\noindent Theorem~\ref{theo:Cal:eq} is then proved.
\end{proof}

\begin{proof}[Proof of Theorem~\ref{thm:CAP:eigen}]
We first observe that both properties of being completely algebraically periodic and being a Prym eigenform is invariant 
along the leaves of the kernel foliation in the Prym loci given in Table~\ref{tab:strata:list}. In view of Theorem~\ref{theo:hyp:Prym:eig}, 
we will show that there exists in the leaf of the kernel foliation through $(X,\omega)$ a surface whose Veech group contains a hyperbolic element.

By normalizing using $\GL^+(2,\R)$ and moving in the kernel foliation leaf, we can suppose that all  the periods of 
$(X,\omega)$ belong to $K(\imath)$. If $K=\Q$, then the $\GL^+(2,\R)$-orbit of 
$(X,\omega)$ has a square-tiled surface. Thus the Veech group of $(X,\omega)$ contains a hyperbolic element, and we are done.

Now assume that $K$ is a real quadratic field. By Theorem~\ref{thm:CAP:to:CP}, we know that $(X,\omega)$ is completely periodic. We can assume that the horizontal and vertical directions are periodic. We want to find a suitable vector $v=(s,t) \in K^2$ such that the Veech group of $(X,\omega)+v$ has two parabolic elements, one preserves the horizontal direction, the other preserves the vertical direction, a product of some powers of such elements provides us with a hyperbolic element in $\SL(X,\omega)$.

Let $C_i, i=1,\dots,k$,  denote the horizontal cylinders, the width, height and modulus of $C_i$ are denoted by 
$w_i,h_i$, and $\mu_i$ respectively. Let $n$ be the number of cylinders up to involution, we choose the numbering of cylinders such that for every $i=1,\dots,n$, if $C_i$ and $C_j$ are permuted by the Prym involution then either $j=i$, or $j>n$. Theorem~\ref{theo:Cal:eq} implies

\begin{equation}\label{eq:CAP:mod:rel}
\sum_{i=1}^{k}w'_ih_i=\sum_{i=1}^{n}\alpha_i\mu_iN(w_i)=0, 
\end{equation}

\noindent where $\alpha_i=1$ if $C_i$ is preserved by the Prym involution, $\alpha_i=2$ otherwise, and $N(w_i)=w_iw'_i\in \Q$.  Remark that for all the Prym loci in Table~\ref{tab:strata:list}, we have $n\leq 3$. By Lemma~\ref{lm:max:nb:cyl}, if this number is maximal ({\em i.e.} equal to three) then the cylinder decomposition is stable ({\em i.e.} every saddle connection in this direction connects a zero to itself). In the case  $n\leq 2$, Equation (\ref{eq:CAP:mod:rel}) implies that all the cylinders are commensurable, therefore there exists a parabolic element in $\SL(X,\omega)$ that fixes the vector $(1,0)$.

Assume that $n=3$, since the cylinder decomposition is stable, in each cylinder the upper (resp. lower) boundary contains only one zero of $\omega$.
For $t \in \R$ such that $|t|$ small enough, the surface $(X,\omega)+(0,t)$ also admits a cylinder decomposition in the horizontal direction with the same topological properties as the decomposition of $(X,\omega)$. Let $C_i^t$ denote the cylinder in $(X_t,\omega_t)=(X,\omega)+(0,t)$ corresponding to $C_i=C^0_i$. Note that $w(C_i^t)=w(C_i)=w_i$ for any $t$, but in general $h_i(t)=h(C_i^t)$ is a non-constant function of $t$. Namely, if the zeros in the upper and lower boundaries of $C_i$ are the same then $h_i(t)=h_i, \; \forall t$, otherwise, either $h_i(t)=h_i \pm t$, or $h_i(t)=\pm t/2$.  

\noindent In particular, we see that $h_i(t)=h_i+\alpha_it$, where $\alpha_i\in \{-1,-1/2,0,1/2,1\}$. There always exist two cylinders $C_i,C_j$ such that  $\alpha_i\neq \alpha_j$.  Set $R_{ij}(t):=\mu(C^t_i)/\mu(C^t_j)$. We have

\begin{equation*}
R_{ij}(t)\in \left\{
\frac{w_j(h_i+t)}{w_ih_j}, \frac{w_j(h_i+t)}{w_i(h_j-t)}, \frac{w_j(h_i+t)}{w_i(h_j+t/2)}, \frac{w_j(h_i+t)}{w_i(h_j-t/2)}, \frac{w_j(h_i+t/2)}{w_ih_j}, \frac{w_j(h_i+t/2)}{w_i(h_j-t/2)} \right\}.
\end{equation*}

\noindent One can easily see that there always exists $t\in K$ such that $R_{ij}(t) \in \Q$.  For $t$ small enough, the surface $(X,\omega)+(0,t)$ is also decomposed into $k$ cylinders in the horizontal direction, and Equation (\ref{eq:CAP:mod:rel}) holds, thus the condition $R_{ij}(t)\in \Q$ implies that all the horizontal cylinders of $(X,\omega)+(0,t)$ are commensurable, which means that  $\SL((X,\omega)+(0,t))$ contains a parabolic element preserving the vector $(1,0)$.

Observe that the horizontal direction on $(X,\omega)+(s,t)$ (for small $s$) is still a periodic direction. 
Thus by the same arguments, we can conclude that there exists a vector $v=(s,t)\in K^2$ such that $\SL((X,\omega)+v)$ contains a hyperbolic element. It follows from Theorem~\ref{theo:hyp:Prym:eig} that $(X,\omega)+v$ is a Prym eigenform, and so is $(X,\omega)$. Theorem~\ref{thm:CAP:eigen} is then proven.
\end{proof}

\section{Complete periodicity of quadratic differentials with periods in a quadratic field}
\label{sec:d=6}

In this section, we will concentrate on cases $(4)-(5)-(6)-(7)-(8)$ of Table~\ref{tab:strata:list}.
Let $\mathcal{Q}(\kappa)$ be one of the following strata 
$$
\QQQ(\kappa) \in \{\QQQ(-1^{4},4),\ \QQQ(-1^3,1,2),\ \QQQ(-1,2,3),\ \QQQ(8),\ \QQQ(-1,1,4)\}.
$$

\noindent We will deduce Theorem~\ref{theo:cp:stronger} from the following:

\begin{Theorem}[$d=6$]
\label{theo:d=6}
Let $T=(\pi,\lbd)$ be a linear involution defined over $6$ intervals which is obtained 
by the cross section of the vertical foliation on a quadratic differential $(Y,q) \in \QQQ(\kappa)$. We assume that the lengths of the intervals exchanged by $T$ belong to a vector space of rank two over $\Q$, and $SAF(T)=0$.
\begin{enumerate}
\item If $(Y,q)\in\QQQ(-1,2,3)$ or $\QQQ(8)$ then $T$ is completely periodic,
\item Otherwise, if $T$ is not completely periodic then:
\begin{enumerate}
\item if $(Y,q) \in \QQQ(-1^3,1,2)$ or $\QQQ(-1^4,4)$ then $(Y,q)$ is the connected sum of a 
flat torus and a flat sphere, irrationally foliated with opposite SAF-invariants.
\item if $(Y,q) \in \QQQ(-1,1,4)$ then $(Y,q)$ is the connected sum of
two flat tori, irrationally foliated with opposite SAF-invariants.
\end{enumerate}
\end{enumerate}
\end{Theorem}

We emphasize that point~$(1)$ of Theorem~\ref{theo:d=6} is false in general.
For instance, for $(Y,q)\in \QQQ(-1^{4},4)$ or $(Y,q)\in \QQQ(-1^{3},1,2)$, one 
may have $SAF(T)=0$ whereas $T$ is not completely periodic, as shown in Figure~\ref{fig:decomposition}.
Indeed, for instance on the left figure, since there is one vertical saddle connection
$(\alpha)$, the permutation associated to $T$ is 
$\left(\begin{smallmatrix} 0& 0 & 1 & 1 & 3 & 4 \\ 2 & 2 & & & 4 & 3 \end{smallmatrix}\right)$.
Clearly one can arrange lengths of the intervals so that $(Y_1,q_1)$ is the connected sum of a 
flat torus (together with a slit) and a flat sphere (with a marked point), 
irrationally foliated with opposite SAF-invariant. The same is true for $(Y_2,q_2)$ (one can 
also give examples for the stratum $\QQQ(-1,1,4)$).
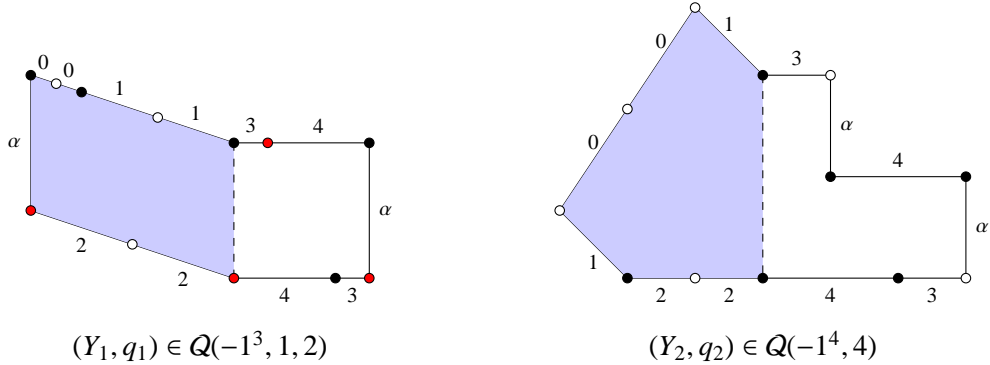
\begin{figure}[htbp]
\begin{minipage}[t]{0.23\linewidth}
\centering
\begin{tikzpicture}[scale=0.9, >=angle 60]
\draw (0,0) -- (0,2) -- (3,1) -- (5,1) -- (5,-1) -- (3,-1) -- (0,0);
\fill[fill=blue!20] (0,0) -- (0,2) -- (3,1) -- (3,-1) -- (0,0);
\draw[-, dashed] (3,1) -- (3,-1);
\foreach \t in {1/2} \filldraw[fill=white] (0,0) + (3*\t,-\t) circle(2pt);
\foreach \t in {1/8,5/8} \filldraw[fill=white] (0,2) + (3*\t,-\t) circle(2pt);
\foreach \t in {0,2/8,1} \filldraw[fill=black] (0,2) + (3*\t,-\t) circle(2pt);
\foreach \x in {(5,1), (4.5,-1)} \filldraw[fill=black] \x circle(2pt);
\foreach \t in {0,1} \filldraw[fill=red] (0,0) + (3*\t,-\t) circle(2pt);
\foreach \x in {(5,-1), (3.5,1)} \filldraw[fill=red] \x circle(2pt);
\draw (0,1) node[left] {$\scriptstyle \alpha$};
\draw (5,0) node[right] {$\scriptstyle \alpha$};

\foreach \t in {1/16,3/16} \draw (0,2) + (3*\t,-\t) node[above] {$\scriptstyle 0$};
\foreach \t in {7/16,13/16} \draw (0,2) + (3*\t,-\t) node[above] {$\scriptstyle 1$};
\foreach \t in {1/4,3/4} \draw (0,0) + (3*\t,-\t) node[below] {$\scriptstyle 2$};

\draw (3.25,1) node[above] {$\scriptstyle 3$};
\draw (4.75,-1) node[below] {$\scriptstyle 3$};
\draw (4.25,1) node[above] {$\scriptstyle 4$};
\draw (3.75,-1) node[below] {$\scriptstyle 4$};

\draw (2.5,-2) node {$(Y_1,q_1) \in \QQQ(-1^3,1,2)$};
\end{tikzpicture}
\end{minipage}
\hskip 20mm
\begin{minipage}[t]{0.6\linewidth}
\centering
\begin{tikzpicture}[scale=0.9, >=angle 60]
\draw (0,0) -- (2,3) -- (3,2) -- (4,2) -- (4,1/2) -- (6,1/2) -- (6,-1) -- (1,-1) -- (0,0);
\fill[fill=blue!20] (0,0) -- (2,3) -- (3,2) -- (3,-1) -- (1,-1) -- (0,0);
\draw[-, dashed] (3,2) -- (3,-1);
\foreach \t in {0,1/2,1} \filldraw[fill=white] (0,0) + (2*\t,3*\t) circle(2pt);
\foreach \x in {(6,-1),(4,2),(2,-1)} \filldraw[fill=white] \x circle(2pt);
\foreach \x in {(1,-1),(3,2),(4,1/2),(6,1/2),(3,-1),(5,-1)} \filldraw[fill=black] \x circle(2pt);

\foreach \x in {(4,1.25),(6,-0.25)} \draw \x node[right] {$\scriptstyle \alpha$};
\foreach \t in {1/4,3/4} \draw (0,0) + (2*\t,3*\t) node[above] {$\scriptstyle 0$};
\foreach \t in {1/2} \draw (2,3) + (\t,-\t) node[above] {$\scriptstyle 1$};
\foreach \t in {1/2} \draw (\t,-\t) node[below] {$\scriptstyle 1$};
\foreach \x in {(1.5,-1),(2.5,-1)} \draw \x node[below] {$\scriptstyle 2$};

\draw (3.5,2) node[above] {$\scriptstyle 3$};
\draw (5.5,-1) node[below] {$\scriptstyle 3$};
\draw (5,1/2) node[above] {$\scriptstyle 4$};
\draw (4,-1) node[below] {$\scriptstyle 4$};
\draw (3,-2) node {$(Y_2,q_2) \in \QQQ(-1^4,4)$};
\end{tikzpicture}
\end{minipage}

\caption{Decompositions of $(Y_i,q_i)$ in a connected sum of a 
flat torus and a flat sphere (colored in blue).
}
\label{fig:decomposition}
\end{figure}

\begin{proof}[Proof of Theorem~\ref{theo:d=6}]
We will assume that $T$ is not completely periodic. Proposition~\ref{prop:connection} implies $T$ has a connection, and by 
Proposition~\ref{prop:BL} only two possibilities can occur. Namely, up to replacing 
$T$ by $\mathcal R^{(n)}(T)$ for some suitable $n$, one has either
\begin{itemize}
\item[(a)] the Rauzy induction is not well defined for $T$, or
\item[(b)] $T$ is decomposed. 
\end{itemize}
First assume case~(a) holds. We get new linear involution $T'=(\pi',\lbd')=\mathcal R_{sing}(T)$
defined over $5$ intervals with zero SAF-invariant. One wants to apply Lemma~\ref{lm:d=5}.
For that, we need to rule out the cases $\pi'\in\mathcal E_{i}$. We begin 
by the following observation. Since $\pi$ is geometrically irreducible,
$\pi'$ is a generalized permutation (and not a ``true'' permutation). 
In particular $\pi' \not \in \mathcal E_{1}$. 
Hence, in view of Lemma~\ref{lm:d=5}, we assume that
$\pi' \in \mathcal E_{2}\cup \mathcal E_{3}$.
\begin{Claim}
\label{claim:1}
If $\pi\in \mathcal R^{-1}_{sing} \mathcal E_{2}$ then $(Y,q)\in \QQQ(-1,1,4)$. 
Moreover if $T$ is not completely periodic then $(Y,q)$ is the connected sum of 
two flat tori, irrationally foliated with opposite SAF-invariant.

\end{Claim}

\begin{proof}
Up to permutation of the letters of the alphabet $\{A,B,C,D,\alpha,\beta\}$ we have:
\begin{multline*}
\mathcal R^{-1}_{sing} \left(\begin{smallmatrix} \alpha & B& C& D & \alpha \\ B& A & D & C & A  \end{smallmatrix}\right)
= \{
\boxed{\left(\begin{smallmatrix} \alpha & B& C& D & \alpha & \beta \\ \beta& A & D & C & A & B\end{smallmatrix}\right)}, 
\boxed{\left(\begin{smallmatrix} \alpha & B& C& D & \alpha & \beta \\ B& \beta & D & C & A & A\end{smallmatrix}\right)}, 
\left(\begin{smallmatrix} \alpha & B&|& C& D & \alpha & \beta \\ B& A&| & \beta & C & A &D \end{smallmatrix}\right), 
\left(\begin{smallmatrix} \alpha & B&|& C& D & \alpha & \beta \\ B& A&| & D & \beta & A &C \end{smallmatrix}\right), 
\left(\begin{smallmatrix} \alpha & B&|& C& D & \alpha & \beta \\ B& A&| & D & C & \beta &A \end{smallmatrix}\right),\\
\left(\begin{smallmatrix} \alpha & B&|& C& D & \alpha & \beta \\ B& A&| & D & C & A &\beta \end{smallmatrix}\right), 
\left(\begin{smallmatrix} \beta & B&|& C& D & \alpha & \beta \\ B& A&| & D & C & A & \alpha \end{smallmatrix}\right), 
\left(\begin{smallmatrix} \alpha & \beta&|& C& D&| & \alpha & \beta \\ B& A&| & D & C&| & A & B \end{smallmatrix}\right), 
\left(\begin{smallmatrix} \alpha & B& \beta& D &|& \alpha & \beta \\ B& A & D & C &|& A &C \end{smallmatrix}\right), 
\left(\begin{smallmatrix} \alpha & B& C& \beta &|& \alpha & \beta \\ B& A & D & C &|& A & D \end{smallmatrix}\right), 
\boxed{\left(\begin{smallmatrix} \alpha & B& C& D & \beta & \beta \\ B& A & D & C & A & \alpha \end{smallmatrix}\right)}
\}.
\end{multline*}
The boxed permutations correspond exactly to irreducible permutations. We also indicate the decomposition
when the permutation is not irreducible (see Definition~\ref{def:irred}). For the above three 
irreducible permutations, the corresponding suspension $(Y,q)$ belongs to the stratum $\QQQ(-1,1,4)$. Moreover, 
if the flow is not completely periodic then the structure of the permutation gives the topological 
decomposition of the surface $Y$. Then claim is proved.
\end{proof}
\begin{Claim}
\label{claim:2}
If $\pi\in \mathcal R^{-1}_{sing} \mathcal E_{3}$ then the followings hold:
\begin{enumerate}
\item $(Y,q) \not \in \QQQ(8)$.
\item If $(Y,q) \in \QQQ(-1,2,3)$ then $\mathcal R_{sing}\pi = \left(\begin{smallmatrix}
C & D & C & A & B \\ E & D & E & B & A \end{smallmatrix}\right)$.
\item If $(Y,q) \in \QQQ(-1^3,1,2) \sqcup \QQQ(-1^4,4)$ then $(Y,q)$ is the connected sum of a 
flat torus and a flat sphere, irrationally foliated with opposite SAF-invariants.
\item If $(Y,q) \in \QQQ(-1,1,4)$ then $(Y,q)$ is the connected sum of
two flat tori, irrationally foliated with opposite SAF-invariants.
\end{enumerate}
\end{Claim}
\begin{proof}
Recall that $\mathcal E_3 = 
\left\{
\left(\begin{smallmatrix}
\begin{smallmatrix}A\ B \\  B\ A  \end{smallmatrix} &\ \pi''
\end{smallmatrix}\right),
\left(\begin{smallmatrix}
\ \pi'' & \begin{smallmatrix}A\ B \\  B\ A  \end{smallmatrix}
\end{smallmatrix}\right)
\right\}$ where $\pi''$ is a generalized permutation defined over $3$ letters. We used the 
same convention as in the proof of Claim~\ref{claim:1}. 

If $\mathcal R_{sing}(\pi)=\left(\begin{smallmatrix} \begin{smallmatrix}A\ B \\  B\ A  \end{smallmatrix} &\ \pi'' \end{smallmatrix}\right)$
then
$$
\pi\in \left\{
\left(\begin{smallmatrix}
\begin{smallmatrix} A \ B &| \\  B\ A &| \end{smallmatrix} &\ \dots  
\end{smallmatrix}\right),
\left(\begin{smallmatrix}
\begin{smallmatrix}\alpha \ B &| \\  B\ A &| \end{smallmatrix} &\ \pi'' & \begin{smallmatrix} | & \alpha \\ |& A  \end{smallmatrix}
\end{smallmatrix}\right),
\boxed{
\left(\begin{smallmatrix}
\begin{smallmatrix}  A \ B \\   B\ \alpha \end{smallmatrix} &\ \pi'' & \begin{smallmatrix} \alpha \\  A  \end{smallmatrix}
\end{smallmatrix}\right)
}
\ \right\}
$$
(up to permutation of the letters of the alphabet $\Alp \sqcup \{\alpha\}$ and of the top/bottom lines). It is 
straightforward to check the conclusions of the claim, for all the possible $\pi''$. 

Now if $\mathcal R_{sing}(\pi)=\left(\begin{smallmatrix} \ \pi'' & \begin{smallmatrix}A\ B \\  B\ A  \end{smallmatrix} 
\end{smallmatrix}\right)$ the proof follows the same lines as the arguments above.
\end{proof}
We can now finish case~(a). Indeed, the only thing remains to prove is when 
$(Y,q) \in \QQQ(-1,2,3)$. But the last claim ensures that $\pi'=\mathcal R_{sing}\pi = 
\left(\begin{smallmatrix} C & D & C & A & B \\ E & D & E & B & A \end{smallmatrix}\right)$.
Thus the only possibility to have $SAF(T')=0$ is $SAF(T')=SAF(T'_1)+SAF(T'_2)=0$ where 
$T'_i=(\pi'_i,\lbd'_i)$ and
$$
\pi'_1 = \left(\begin{smallmatrix} C & D & C \\ E & D & E \end{smallmatrix}\right),
\qquad
\pi'_2 = \left(\begin{smallmatrix} A & B \\ B & A \end{smallmatrix}\right).
$$
But obviously $T'_1=(\pi'_1,\lbd'_1)$ is completely periodic for any $\lbd'_1$.
Hence $SAF(T'_1)$ vanishes and so does $SAF(T'_2)$, which implies that $T'_2$ is completely periodic. Therefore 
$T$ is periodic. \medskip

\noindent In case~(b), $T=T_{1}\comp T_{2}$, where $T_{i}$ is a linear involution
defined over $d_{i}$ intervals with $d_{1}+d_{2}=6$. Since $\pi$ is irreducible the only possible 
partitions for $\{d_{1},d_2\}$ are $\{2,4\}$ or $\{3,3\}$. 
Since $SAF(T_{1}) + SAF(T_{2})=0$ we easily rule out the first case (otherwise $T$ is completely periodic).
In the second case $\pi = (\pi_1,\pi_2)$, where $\pi_1$ and $\pi_2$ are generalized permutations, each defined 
over $3$ letters. \medskip

We first claim that $(Y,q) \not \in \QQQ(8)  \sqcup \QQQ(-1,2,3) \sqcup \QQQ(-1^3,1,2)$. The proof is 
straightforward (for each possible permutation, we calculate the stratum corresponding to $\pi$). \medskip

Now if $(Y,q) \in \QQQ(-1^{4},4) \sqcup \QQQ(-1,1,4)$, since $T$ is decomposed, we have a decomposition of the surface $(Y,q)$. Note that if $T$ is not periodic then $SAF(T_{1}) = - SAF(T_{2}) \not = 0$.
Then necessarily $(Y,q)$ is the connected sum of a flat torus and a flat sphere
(respectively, two flat tori) irrationally foliated with opposite SAF-invariant (see Figure~\ref{fig:decomposition:2} below).
This ends the proof of Theorem~\ref{theo:d=6}.
\end{proof}
\begin{Example}
\label{ex:H22}
In Figure~\ref{fig:decomposition:2} below the surface $(Y,q)$ decomposes
along the saddle connection $\gamma$ into a connected sum of a 
flat torus and a flat sphere, as we can notice by the underlying permutation
$\left(\begin{smallmatrix} 0 &0&1&3&4&3&4 \\ 1 &2&2&5&5 \end{smallmatrix}\right)$.
One can arrange the parameters so that $SAF(T)=0$ and $T$ is not completely
periodic. Moreover, there exists a regular fixed point of the Prym involution in the spine of the double cover $(X,\omega)$. Namely $(X,\omega)$ decomposes
into two permuted tori and one invariant tori along the lifts of $\gamma,\gamma'$.
\begin{figure}[htbp]
\begin{minipage}[t]{0.8\linewidth}
\centering
\begin{tikzpicture}[scale=0.9, >=angle 60]
\draw (0,0) -- (2,3) -- (3,2) -- (4.5,2) -- (5.5,1/2) -- (7,1/2) -- (8,-1) -- (1,-1) -- (0,0);
\fill[fill=blue!20] (0,0) -- (2,3) -- (3,2) -- (3,-1) -- (1,-1) -- (0,0);
\draw[-, dashed] (3,2) -- (3,-1);
\draw[-, dashed] (5.5,1/2) -- (5.5,-1);
\foreach \t in {0,1/2,1} \filldraw[fill=white] (0,0) + (2*\t,3*\t) circle(2pt);
\filldraw[fill=white] (2,-1) circle(2pt);
\filldraw[fill=white] (5.5,-1) circle(2pt);
\foreach \x in {(1,-1),(3,2),(4.5,2),(5.5,1/2),(7,1/2),(8,-1),(3,-1)} \filldraw[fill=black] \x circle(2pt);

\foreach \x in {(5,1.25),(7.5,-0.25)} \draw \x node[right] {$\scriptstyle \alpha$};
\foreach \t in {1/4,3/4} \draw (0,0) + (2*\t,3*\t) node[above] {$\scriptstyle 0$};
\foreach \t in {1/2} \draw (2,3) + (\t,-\t) node[above] {$\scriptstyle 1$};
\foreach \t in {1/2} \draw (\t,-\t) node[below] {$\scriptstyle 1$};
\foreach \x in {(1.5,-1),(2.5,-1)} \draw \x node[below] {$\scriptstyle 2$};

\draw (3.75,2) node[above] {$\scriptstyle 3$};
\draw (6.75,-1) node[below] {$\scriptstyle 4$};
\draw (6.25,1/2) node[above] {$\scriptstyle 3$};
\draw (4.25,-1) node[below] {$\scriptstyle 4$};
\draw (3,0.5) node[right] {$\scriptstyle \gamma$};
\draw (5.5,-0.25) node[right] {$\scriptstyle \gamma'$};
\draw (4.5,-2) node {$(Y,q) \in \QQQ(-1^4,4)$};
\end{tikzpicture}
\end{minipage}
\caption{Decompositions of $(Y,q)$ into a connected sum of a 
flat torus and a flat sphere (colored in blue).
}
\label{fig:decomposition:2}
\end{figure}
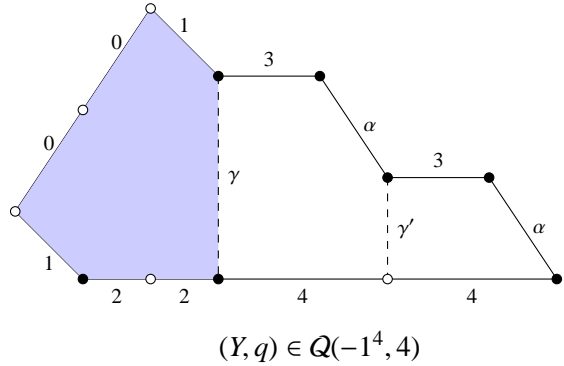
\end{Example}
In view of proving Theorems~\ref{theo:cp:stronger} and~\ref{theo:infinite:limitset} 
we reformulate Theorem~\ref{theo:d=6} in terms of Prym eigenforms.
\begin{Corollary}
\label{cor:reformulation:d=6}
Let $(X,\omega)\in \Omega E_D(\kappa)$ be a Prym eigenform and assume that $\omega$ has 
absolute and relative periods in $K(\imath)$, where $K=\Q(\sqrt{D})$. Let $\theta$ be a direction in $K\P^1$.
\begin{enumerate}
\item If $(X,\omega)\in \PrD(1,1,4) \sqcup  \PrD(4,4)$ then $\mathcal{F}_\theta$ is completely periodic.
\item 
\label{pt:2}
If $(X,\omega)\in \PrD(1,1,2)$ and the spine of the foliation in direction $\theta$
contains a regular fixed point of the Prym involution then $\mathcal{F}_\theta$ is completely periodic.
\end{enumerate}
\end{Corollary}
\noindent Again we emphasize that assertion~(\ref{pt:2}) of above corollary is false in general 
if $(X,\omega)\in \PrD(2,2)$, as shown Example~\ref{ex:H22} (see also Figure~\ref{fig:decomposition:2}).

\begin{proof}[Proof of the corollary]
As usual we will assume that $\theta$ is vertical. We begin by observing that 
Theorem~\ref{thm:Alg:Periodic} implies $SAF(X,\omega)=0$. Let $T$ be the 
cross section of the vertical foliation to some full transversal interval on the 
quotient $(Y,q)=(X,\omega)/<\inv>$. Hence $SAF(T)=0$. Complete periodicity 
of $\mathcal{F}_\theta$ is equivalent to complete periodicity of $T$. \medskip

\noindent Since $(X,\omega)\in \Prym(1,1,4)$ (respectively, $(X,\omega)\in \Prym(4,4)$) 
is equivalent to $(Y,q)\in \QQQ(-1,2,3)$ (respectively, 
$(Y,q)\in \QQQ(8)$), assertion~$(1)$ is a reformulation of Theorem~\ref{theo:d=6}. \medskip

\noindent Let us prove~$(2)$. Again $(X,\omega)\in \Prym(1,1,2)$
is equivalent to  $(Y,q)\in \QQQ(-1^3,1,2)$. If $T$ is not completely periodic then by 
Theorem~\ref{theo:d=6} $(Y,q)$ is the connected sum of a flat torus and a flat 
sphere, irrationally foliated with opposite SAF-invariants. Hence there exists two 
(\^homologous) saddle connections $\gamma$ and $\gamma'$ such that 
$Y\setminus (\gamma \cup \gamma') = Y_0 \cup Y_1$, where $Y_0$ is a flat 
sphere and $Y_1$ a flat torus (with boundary). Observe that the three poles of $q$
belong to $Y_0$. Since $Y\in \QQQ(-1^3,1,2)$ the component $Y_0$ lifts  
to a fixed torus $X_0$ and $Y_1$ to two permuted tori $X_{1,j}, \, j=1,2,$ in $X$. One has 
$SAF(X_0) = -2SAF(X_{0,1})$. By assumption the spine of the foliation on $X_0$
contains a fixed point of the Prym involution; hence $\mathcal {F_\theta}_{|X_0}$ is periodic. Thus 
$SAF(X_0)=0$ and we conclude that $\mathcal {F_\theta}$ is periodic on $X$.
\end{proof}

\begin{Remark}
We note that this proof does not work if $(Y,q)\in \QQQ(-1^4,4)$. Indeed the decomposition 
into three tori $T_0$ and $T_{0,j}$ still holds but it could happen that the pole on the 
spine of the foliation on $Y$ lies in $\gamma \cup \gamma'$ (see Figure~\ref{fig:decomposition:2}).
In this case, the foliation on $Y_0$ may not be periodic, for instance, in Example~\ref{ex:H22}, if we choose the lengths $\lbd_\alpha,\lbd_3$
of the intervals labelled by $\alpha$ and $3$ so that $\frac{\lbd_\alpha}{\lbd_3}\not \in\Q$ .
\end{Remark}

\begin{proof}[Proof of Theorem~\ref{theo:cp:stronger}]
It clearly follows from  assertion~(1) of Corollary~\ref{cor:reformulation:d=6}.
\end{proof}

\section{Limit set of Veech groups}
\label{sec:limit:sets}

In this section, we prove the result on the limit sets of Veech groups
of Prym eigenforms {\em i.e.} Theorem~\ref{theo:infinite:limitset}. In the sequel
we fix a form $(X,\omega)\in \Omega E_D(4,4)^{\rm even} \sqcup \Omega E_D(1,1,4) \sqcup \Omega E_D(1,1,2)$.
A periodic direction is said to be {\em stable} if 
there is no saddle connection in this direction that connects two different zeros, 
it is said to be {\em unstable} otherwise.
\begin{Lemma}\label{lm:max:nb:cyl}
Any direction $\theta$ that decomposes $(X,\omega)\in \H(\kappa)$ into $g+|\kappa|-1$ cylinders, where $g$ is the genus of $X$, is stable.
\end{Lemma}

\begin{proof}
We begin by observing that any periodic direction decomposes the surface $X$ into {\it at most}
$g+|\kappa|-1$ cylinders~\cite{Ma}. Now if the direction $\theta$ is not stable then there
exists necessarily a saddle connection between two different zeros that we can
collapse to a point (in direction $\theta$) without destroying any cylinder. But in this way we get a surface $(X',\omega')\in \H(\kappa')$
of genus $g$ where $|\kappa'| < |\kappa|$, and having $g+|\kappa|-1$ cylinders . This is a contradiction.
\end{proof}
We now prove the following proposition about unstable periodic directions on Prym eigenforms. 
\begin{Proposition}
\label{prop:cp:unstab:decomp}
Let $(X,\omega)\in \Omega E_D(4,4)^{\rm even} \sqcup \Omega E_D(1,1,4) \sqcup \Omega E_D(1,1,2) $  
be a Prym eigenform with relative and absolute periods in $K(\imath)$, where $K=\Q(\sqrt{D})$. Then any unstable periodic direction 
$\theta$ decomposes the surface into cylinders with commensurable moduli. 
As a consequence, $\SL(X,\omega)$ contains a parabolic element fixing $\theta$.
\end{Proposition}
\begin{proof}
We only need to consider the case where $K$ is a real quadratic field. Assume that $(X,\omega)\in \Omega E_D(4,4)^{\rm even} \sqcup \Omega E_D(1,1,4)$ 
then the decomposition in direction $\theta$ has at most $6$ cylinders by Lemma~\ref{lm:max:nb:cyl}.  Since the direction $\theta$ is not stable and none of the cylinders is fixed by the Prym involution (otherwise the quotient $(Y,q)$ by the Prym involution would have at least $2$ poles) one has $n\in \{2,4\}$. We denote by $C_i, i=1,\dots,2r=n,$ the permuted cylinders so that $C_{i+r}=\inv(C_i)$.
By Corollary~\ref{cor:mod:relation} the moduli of the cylinders satisfies
$$
0=\mathrm{Im}(\Fl(\omega)) = 2 \sum_{i=1}^{r} h(C_{i}) w(C_{i})' = 2 \sum_{i=1}^{r} k_{i}\cdot \mu(C_{i}).
$$
where $k_{i}=w(C_{i})w(C_{i})'\in\Q$. Hence
$$
\sum_{i=1}^{r} k_{i}\cdot \mu(C_{i}) = 0.
$$
But $r\leq 2$ thus above equality implies that $\mu(C_i)$ are commensurable. The
direction is parabolic and a suitable product of Dehn twist in each cylinder gives rise to an affine automorphism
with parabolic derivative fixing $\theta$.

The case $(X,\omega)\in \Omega E_D(1,1,2)$ follows from similar arguments since 
the decomposition in direction $\theta$ has at most $5$ cylinders.
\end{proof}

\begin{Corollary}
\label{cor:P112:2zero:connect}
Let $(X,\omega)\in  \Omega E_D(1,1,2)$  be a  Prym eigenform
with relative and absolute periods in $K(\imath)$.
If $\theta \in K\P^1$ is the direction of a saddle connection between the two simple zeros that is
invariant under the Prym involution, then $\SL(X,\omega)$ contains a parabolic element fixing $\theta$.
\end{Corollary}

\begin{proof}[Proof of Corollary~\ref{cor:P112:2zero:connect}]
In view of the previous proposition it suffices to show that $\theta$ is an unstable periodic direction.
Since $\theta$ is the direction of a saddle connection, we have $\theta\in K\P^1$. Necessarily the saddle connection contains a regular fixed point of the Prym involution. By Corollary~\ref{cor:reformulation:d=6}, assertion~(2), the flow $\mathcal F_\theta$ is completely periodic
(the spine contains a regular fixed point of the Prym involution).  Since there is a saddle connection connecting two different zeros, this periodic direction is unstable, and the corollary follows.
\end{proof}

\subsection{Proof of Theorem~\ref{theo:infinite:limitset},  Case $\Prym(4,4)^{\rm even}\sqcup\Prym(1,1,4)$}

\begin{proof}
If the limit set has at least two points then there is a hyperbolic element in $\SL(X,\omega)$ represented by an
affine pseudo-Anosov homeomorphism $\phi$. By a result of McMullen (\cite{Mc2}, Theorem 9.4) we can assume that all the periods of $\omega$ belong to $K(\imath)$.

By Theorem~\ref{theo:cp:stronger} and Proposition~\ref{prop:cp:unstab:decomp}, any linear foliation on $(X,\omega)$ in the direction $\theta$ of  a saddle connection between two different zeros is fixed by a parabolic element of $\SL(X,\omega)$.  It remains to show that those directions fill out a dense subset of $\R\P^{1}$, which implies  that the limit set is the full circle at infinity.

Let $\theta_0\in\R\P^1$ and fix $\varepsilon>0$. By Theorem~\ref{theo:cp:stronger}, one can find $\theta\in K\P^1$ so that foliation $\mathcal F_\theta$
is completely periodic and $|\theta-\theta_0|<\varepsilon/2$. If the direction $\theta$ is not stable then by Proposition~\ref{prop:cp:unstab:decomp} we are done. Otherwise $X$ is decomposed into $6$ cylinders in direction $\theta$. Since
$X$ is a connected surface, we claim that there exists a cylinder $C$ such that the top boundary of $C$ is made
of saddle connections between one zero $P$ and the bottom boundary is made
of saddle connections between one other zero $Q \neq P$. By a suitable Dehn twist,
it is easy to find a new direction $\theta'$ satisfying $|\theta-\theta'|< \varepsilon/2$ such that there is
a saddle connection contained in $C$ between $P$ and $Q$ in direction $\theta'$. This is the desired direction.
\end{proof}

\subsection{Proof of Theorem~\ref{theo:infinite:limitset}, Case  $\Prym(1,1,2)$}
\begin{proof}
We now show the result for $(X,\omega)\in  \Omega E_D(1,1,2)$. By a result of Masur~\cite{Ma}, we know  that the set $\Theta$ of directions $\theta\in \R\P^1$ such that $\theta$ is  the direction of a regular closed geodesic is dense in $\R\P^1$. Thus, by using Proposition~\ref{prop:cp:unstab:decomp}, it suffices to show that any  direction $\theta\in \Theta$  is contained in the closure of the set of unstable periodic directions.

Let $\theta_0$ be a direction in $\Theta$. By Theorem~\ref{thm:CAP:to:CP}, we know that $X$ is decomposed into cylinders in direction $\theta_0$. We can assume that $\theta_0$ is the horizontal direction.  Obviously, we only need to consider the case where $\theta_0$ is a stable periodic direction. Note that in this case $X$ is decomposed into $5$ cylinders in direction $\theta_0$. 

If $\gamma$ is a geodesic segment connecting a regular fixed point of $X$ to one simple zero of $\omega$, then $\gamma\cup\inv(\gamma)$ is a saddle connection joining two simple zeros and invariant under $\inv$. Following Corollary~\ref{cor:P112:2zero:connect}, the direction of $\gamma$ is an unstable periodic direction. We claim that there exist such geodesic segments whose direction is arbitrarily close to $\theta_0$.

We begin by observing that in a cylinder decomposition of $X$, 
only one cylinder (denoted by $C_0$) is invariant under $\inv$.  Recall that 
$\inv$ has three regular fixed points, two of which are contained in $C_0$, the third one is the midpoint of a saddle connection 
contained in the boundaries of two exchanged cylinders. We can divide those decompositions into three types:

\begin{itemize}
\item[(a)] The boundary of the $C_0$ contains the simple zeros, or
\item[(b)] The boundary of $C_0 $ only contains the double zero, and $C_0$ is a simple cylinder, or
\item[(c)] The boundary of $C_0 $ only contains the double zero, and $C_0$ is not a simple cylinder.
\end{itemize}
(a cylinder is simple if each of its boundary component consists of  exactly one saddle connection). \medskip

In Case (a), there exist saddle connections contained in $C_0$ and invariant under $\inv$  which connect the two simple zeros 
whose direction is arbitrarily close to $\theta_0$.

In Case (b), there exist two cylinders $C_1,C_2$ exchanged by $\inv$ such that the bottom boundary of $C_1$ contains a 
regular fixed point of $\inv$, and the top boundary of $C_1$ contains a simple zero. Therefore, there  exist saddle connections 
contained in $C_1\cup C_2$ joining the simple zeros and invariant under $\inv$ whose direction is arbitrarily close to $\theta_0$. 

In Case (c), the only topological model is presented in Figure~\ref{fig:P112:Veech:case:c} below.
\begin{figure}[htbp]
\centering
\begin{tikzpicture}[scale=0.3]
\fill[fill=yellow!80!black!20] (0,1) -- (0,-1) -- (9,-1) -- (9,1) -- cycle;
\draw (0,1) -- (0,-1) -- (3,-1) -- (3,-6) -- (6,-6) -- (6,-4) -- (9,-4) -- (9,6) -- (6,6) -- (6,4) -- (3,4) -- (3,1) -- cycle; 
\draw (3,1) -- (9,1) (3,-1) -- (9,-1) (6,4) -- (9,4) (3,-4) -- (6,-4);
\draw[red, dashed] (3,-4) -- (9,4);
\foreach \x in {(0,1), (0,-1), (3,1) ,(3,-1), (9,1), (9,-1)} \filldraw[fill=white] \x circle (5pt); 
\foreach \x in {(3,4), (3,-6), (6,-6), (6,4), (9,4)} \filldraw[fill=gray] \x circle (5pt);
\foreach \x in {(3,-4), (6,-4), (6,6), (9,-4), (9,6)} \filldraw[fill=black] \x circle (5pt);
\filldraw[red] (6,0) circle (2pt);  
\end{tikzpicture}
\caption{Stable cylinder decomposition in $\Prym(1,1,2)$, the double zero is colored in white (the ``gray'' cylinder is the unique cylinder invariant under $\inv$).}
\label{fig:P112:Veech:case:c}
\end{figure}
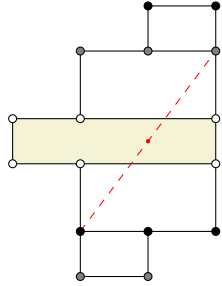
One can easily see that there always exists a geodesic segment from a fixed point of $\inv$ in the interior of $C_0$ to a 
simple zero in the boundary of a cylinder adjacent to $C_0$. Using Dehn twists, we see that there exist infinitely many such 
segments whose direction can be made arbitrarily close to $\theta_0$. The theorem is then  proved for this case.
\end{proof}

\begin{Remark}
\label{rk:alternative}
Actually we also proved a slightly different result: the limit set of the Veech group of 
any $(X,\omega)\in \PrD(4,4)^{\rm even}\sqcup\PrD(1,1,4)\sqcup\PrD(1,1,2)$, having all periods in 
$K(\imath)$, is the full circle at infinity.
\end{Remark}

\subsection{Infinitely generated Veech groups}
We end with the proof of the last claim in Theorem~\ref{theo:infinite:limitset} {\em i.e.}

\begin{Corollary}
There exist infinitely many surfaces in $\Prym(4,4)^{\rm even}\sqcup\Prym(1,1,4)\sqcup\Prym(1,1,2)$ 
whose Veech group is infinitely generated and of the first kind.
\end{Corollary}

\begin{proof}
Recall that a Fuchsian group is said to be of the first kind if its limit set is the full circle 
at infinity. Such a group is either a lattice, or infinitely generated 
(see {\em e.g.}~\cite{Katok}). Hence, in view of Remark~\ref{rk:alternative}, 
it suffices to exhibit Prym eigenforms (with relative periods in $K(\imath)$) whose the Veech group is not a lattice. 
But a theorem of Veech~\cite{Veech1989} asserts that in the lattice case
the directional flow  $\mathcal F_\theta$ is either uniquely ergodic or parabolic ({\em i.e.} 
the surface is decomposed into cylinders of commensurable moduli in direction $\theta$). Hence it suffices to give examples where
\begin{enumerate}
\item $\mathcal F_\theta$ is neither periodic nor dense, or 
\item $\mathcal F_\theta$ is periodic, but with incommensurable cylinders.
\end{enumerate}
\noindent In~\cite{Lanneau:Mahn:ratner}, we will prove a sufficient condition for a surface in $\Omega E_D(1,1,2)$ to be non-Veech. It follows that 
for any quadratic discriminant $D$ where $\PrD(1,1,2) \not = \varnothing$ there 
are infinitely may surfaces $(X,\omega)\in\PrD(1,1,2)$ which are non-Veech. For the sake of completeness we give here explicit examples, which satisfy $(2)$. We only focus on the locus 
$\Prym(1,1,2)$  since similar constructions work for the two other loci.

We begin by choosing a tuple $(w,h,e)$ of integers such that:
$$
\left\{\begin{array}{l}   w>0,   \   h>0,\\  
e+2h<w,\\  
\gcd(w,h,e)=1,  \textrm{  and } D=e^2+8w h.\\
\end{array}
\right.
$$ 
Let $\lbd:=\frac{e+\sqrt{D}}{2} >0$ (remark that $\lbd < w$). We also choose $t\in\Q(\sqrt{D})$
so that $0< t < \lbd$.  Let $(X,\omega)$ be the surface represented in Figure~\ref{fig:prototype} having the following coordinates
$$
\left\{ \begin{array}{ll}
  \omega(\alpha_1) = (\lbd,0) &  \omega(\beta_1) = (0,\lbd) \\
   \omega(\alpha_{2,1}) =    \omega(\alpha_{2,2}) = (w/2,0) &    \omega(\beta_{2,1}) = \omega(\beta_{2,1}) = (0,h/2) \\
   \omega(\eta) = (t,0)
\end{array}
\right. 
$$
\begin{figure}[htbp]
\centering \subfloat{
\begin{tikzpicture}[scale=0.7]
\fill[fill=yellow!80!black!20,even  odd rule]  (0,0)  rectangle (3,3);
\draw (-2,0) -- (-2,-1) -- (1,-1) -- (1,0) -- (3,0)  -- (3,3)  -- (5,3) -- 
(5,4) -- (2,4) -- (2,3) -- (0,3) -- (0,0) -- cycle; 
\draw (0,0) -- (1,0) (2,3) -- (3,3);
\filldraw[fill=white,  draw=black]  (0,0)  circle (2pt)  (0,3)  circle
(2pt) (0,-1) circle (2pt)  (3,0) circle (2pt) (3,3) circle (2pt) (3,4) circle (2pt);
\filldraw[fill=black,  draw=black]  (1,0)  circle (2pt)  (1,3)  circle
(2pt) (1,-1) circle (2pt)  (-2,0) circle (2pt) (-2,-1) circle (2pt) ;
\filldraw[fill=gray,  draw=black]  (2,0)  circle (2pt)  (2,3)  circle
(2pt) (2,4) circle (2pt)  (5,3) circle (2pt) (5,4) circle (2pt) ;
\draw[->,>= angle 45, thin,  dashed] (0,1.5) -- (3,1.5); 
\draw[->,>= angle 45, thin,  dashed] (1.5,0) -- (1.5,3); 
\draw[->,>= angle 45, thin,  dashed] (2,3.5) -- (5,3.5); 
\draw[->,>= angle 45, thin,  dashed] (4,3) -- (4,4); 
\draw[->,>= angle 45, thin,  dashed] (-2,-0.5) -- (1,-0.5); 
\draw[->,>= angle 45, thin,  dashed] (-1,-1) -- (-1,0); 
\draw[thick,  dashed, ->,  >=angle 45] (1,0)  .. controls (1.5,1)  and (1.5,1) ..   (2,0); 
\draw (0.5,1.5) node[above]  {$\scriptstyle \alpha_1$}
  (1.25,2.25) node  {$\scriptstyle \beta_1$}  (2, 3.5) node[left] {$\scriptstyle
        \alpha_{2,1}$} (4,3) node[below] {$\scriptstyle \beta_{2,1}$}
      (-2,-0.5)  node[left]  {$\scriptstyle \alpha_{2,2}$}  (-1,0)
      node[above] {$\scriptstyle \beta_{2,2}$}
    (1.75,0.5) node[right]  {$\scriptstyle \eta$};
\end{tikzpicture}
}
\caption{A translation surface $(X,\omega)\in \Prym(1,1,2)$. The double zero is represented 
in white color (the fixed  cylinder  is colored  in  grey).  The identifications of the sides are 
the ``obvious'' identifications.
}
\label {fig:prototype}
\end{figure}
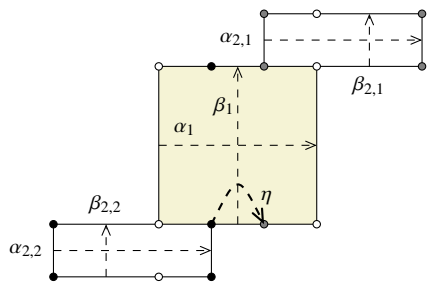
Letting $\alpha_2:=\alpha_{2,1}+\alpha_{2,2}$  and $\beta_2:=\beta_{2,1}+\beta_{2,2}$, 
the set $\{\alpha_i,\beta_i\}_{i=1,2}$ is a  symplectic  basis of $H_{1}(X,\Z)^{-}$. Moreover, in these coordinates 
the   restriction    of   the intersection form is given by the matrix $\displaystyle{\left(%
\begin{smallmatrix}
  J & 0 \\ 0 & 2J \\
\end{smallmatrix}%
\right)}$. In particular it is straightforward to check that the endomorphism $ T=\left(%
\begin{smallmatrix}  e & 0 & w & 0 \\ 0 & e & 0 & h \\ 2h & 0 & 0 & 0 \\ 0 & 2w & 0 & 0 \end{smallmatrix}%
\right)$ (in the basis $(\alpha_i,\beta_i)_{i=1,2}$) is self-adjoint and satisfies $T^2=eT+2w h\textrm{Id}_{\R^{4}}$ and $T^*(\omega)=\lbd\omega$. 
Hence $(X,\omega)\in\PrD(1,1,2)$. 
A straightforward computation shows that the moduli of the cylinders in the vertical direction are given by
$$
\cfrac{t}{\lbd}, \qquad \cfrac{\lbd-2t}{\lbd+h/2} \qquad \textrm{and} \qquad \cfrac{w/2-(\lbd-2t)}{h/2}.
$$
\noindent One can easily see that if $t/\lbd \in \Q$, then the first two moduli are
incommensurable. Hence the Veech group of the corresponding surface $(X,\omega)$ is infinitely generated. This finishes the proof of the corollary.
\end{proof}


\end{document}